\documentclass[12pt]{amsart}%
\usepackage{graphicx,subcaption}
\usepackage{lipsum}
\usepackage{mwe}
\usepackage[small,nohug,heads=littlevee]{diagrams}
\usepackage{tikz}
\usetikzlibrary{matrix,arrows,decorations.pathmorphing}
\usepackage{amscd}
\usepackage{geometry}
\usepackage{amsmath}
\usepackage{amsfonts}
\usepackage{amssymb}%
\usepackage{enumerate}
\usepackage{extpfeil}
\usepackage{cleveref}
\providecommand{\U}[1]{\protect\rule{.1in}{.1in}}
\newtheorem{theorem}{Theorem}[section]
\theoremstyle{plain}

\newtheorem{corollary}[theorem]{Corollary}

\newtheorem{lemma}[theorem]{Lemma}
\newtheorem{question}{Question}
\newtheorem{proposition}[theorem]{Proposition}
\theoremstyle{definition}
\newtheorem{definition}{Definition}[section]
\newtheorem{example}{Example}

\newtheorem{remark}{Remark}

\numberwithin{equation}{section}

\usepackage [autostyle, english = american]{csquotes}
\MakeOuterQuote{"}

\makeatletter
\def\oversortoftilde#1{\mathop{\vbox{\m@th\ialign{##\crcr\noalign{\kern3\p@}%
      \sortoftildefill\crcr\noalign{\kern3\p@\nointerlineskip}%
      $\hfil\displaystyle{#1}\hfil$\crcr}}}\limits}

\def\sortoftildefill{$\m@th \setbox\z@\hbox{$\braceld$}%
  \braceld\leaders\vrule \@height\ht\z@ \@depth\z@\hfill\braceru$}

\makeatother

\begin{document}
\title[$\mathcal{Z}$-compactifiable manifolds which are not pseudo-collarable]{$\mathcal{Z}$-compactifiable manifolds which are not pseudo-collarable}
\author{Shijie Gu}
\address{Department of Mathematical Sciences\\
Central Connecticut State University, New Britain, CT, 06053}
\email{sgu@ccsu.edu}
\thanks{}
\date{September 19th, 2021.}

\begin{abstract}
It is shown that there exist $\mathcal{Z}$-compactifiable manifolds
with noncompact boundary which fail to be pseudo-collarable.
\end{abstract}
\maketitle
\section{Introduction}
An $m$-manifold $M^{m}$ with (possibly empty) boundary is \emph{completable} if there exists a compact manifold $\widehat{M}^{m}$ and a compactum $C\subseteq\partial\widehat{M}^{m}$ such that $\widehat{M}^{m}\backslash C$ is homeomorphic to $M^{m}$. In this case $\widehat{M}^{m}$ is called a \emph{(manifold) completion }of $M^{m}$.\footnote{In literature, when $M^m$ has one end, completing a manifold is also known as finding a collar of the end.} A full characterization for completable manifolds, in the piecewise-linear (PL), topological (TOP) and smooth (DIFF) categories, was achieved through combined efforts of topologists over the past five decades; see \cite{Sie65}, \cite{Tuc74}, \cite{O'B83} and \cite{GG20}.

\begin{theorem}
\emph{[Manifold Completion Theorem]}\label{Th: Completion Theorem} An
$m$-manifold $M^{m}$ ($m\neq 4,5$) is completable if and only if

\begin{enumerate}
\item \label{Char1}$M^{m}$ is inward tame,

\item \label{Char2}$M^{m}$ is peripherally $\pi_{1}$-stable at infinity,

\item \label{Char3}$\sigma_{\infty}(M^{m})\in\underleftarrow{\lim}\left\{
\widetilde{K}_{0}(\pi_{1}(N))\mid N\text{ a clean neighborhood of
infinity}\right\}  $ is zero, and

\item \label{Char4}$\tau_{\infty}\left(  M^{m}\right)  \in\underleftarrow{\lim
}^{1}\left\{  \operatorname*{Wh}(\pi_{1}(N))\mid N\text{ a clean neighborhood
of infinity}\right\}  $ is zero.\bigskip
\end{enumerate}
\end{theorem}
\begin{remark}
When $m =1$, the characterization is trivial. For $m = 2,3$, Condition (\ref{Char1}) alone would be necessary and sufficient. See \cite{GG20} and \cite{Tuc74}. The Manifold completion theorem fails in dimension 4. Counterexamples were discovered in \cite{Wei87} and \cite{KS88}. In TOP, Theorem 1.1 holds in dimension 5 provided Condition
(\ref{Char2}) is strengthened to ensure the existence of arbitrarily small neighborhoods
of infinity with stable peripheral pro-$\pi_1$ groups that are "good" in the language of \cite{FQ90}.
\end{remark}

Although Condition (\ref{Char2}) is necessary in order for such a completion to exist, such condition is too rigid to characterize many exotic examples related to current research trends in topology and geometric group theory. For instance, the exotic universal covering spaces produced by Davis \cite{Dav83} are not completable (because Condition (\ref{Char2}) fails) yet their ends exhibit nice geometric structure. More examples such as (open) manifolds that satisfy Conditions (\ref{Char1}), (\ref{Char3}) and (\ref{Char4}) but not Condition (\ref{Char2}) can be found in \cite[Thm.1.3]{GT03}. Is there a way to characterize manifolds sharing the structure with these manifolds at infinity? When $M^m$ is an open manifold (or more generally, a manifold with compact boundary), by weakening Condition (\ref{Char2}), Guilbault \cite{Gui00} initiated a program to answer this question, which can be viewed as a natural generalization of Siebenmann's collaring theorem for manifolds with compact boundary. Define a manifold neighborhood of infinity $N$ in a manifold $M^m$ to be a \emph{homotopy collar} provided $\operatorname{Fr}N \hookrightarrow N$ is a homotopy equivalence. A \emph{pseudo-collar} is a homotopy collar which contains arbitrarily small homotopy collar neighborhoods of infinity. A manifold is \emph{pseudo-collarable} if it contains a pseudo-collar neighborhood of infinity. The idea of pseudo-collars and a detailed motivation for the definition are nicely exposited in \cite{Gui00}. Through a series of papers \cite{Gui00, GT03, GT06}, a characterzation for pseudo-collarable manifolds with compact boundary (hence, finitely many ends) was provided. In a recent work, the author \cite{Gu20} developed a complete characterization for high-dimensional manifolds with (possible noncompact) boundary.

\begin{theorem}
[Pseudo-collarability characterization theorem]\label{Th: Characterization Theorem} An $m$-manifold $M^m$ ($m\geq 6$) is pseudo-collarable iff each of the
following conditions holds:

\begin{enumerate}[(a)]

\item \label{condition a} $M^{m}$ is inward tame,

\item \label{condition b} $M^{m}$ is peripherally perfectly $\pi_1$-semistable at infinity, and

\item \label{condition c} $\sigma_{\infty}(M^{m})\in\underleftarrow{\lim}\left\{
\widetilde{K}_{0}(\pi_{1}(N))\mid N\text{ a clean neighborhood of
infinity}\right\}  $ is zero. \bigskip
\end{enumerate}
\end{theorem}
\begin{remark}
The dimension restriction $m\geq 6$ is only required for the sufficiency of Theorem \ref{Th: Characterization Theorem}, see \cite[Section 6]{Gu20}. Compared to Theorem \ref{Th: Completion Theorem}, when $m\leq 3$, pseudo-collarability is just the ordinary collarability.
When all of the groups in Condition (\ref{condition b}) involved are "good" in
the sense of Freedman and Quinn, the conclusion holds for $m=5$ in TOP. Just like Theorem 
\ref{Th: Completion Theorem}, Theorem \ref{Th: Characterization Theorem} fails in dimension 4. Counterexamples are again the ones from  \cite{Wei87} and \cite{KS88}.
\end{remark}

Theorem \ref{Th: Characterization Theorem} captures the pseudo-collar structure of the ends, however, it is unclear if pseudo-collarability guarantees a compactification called $\mathcal{Z}$-compactification\footnote{It is an open question that whether pseudo-collarability and Condition (\ref{Char4}) of Theorem \ref{Th: Completion Theorem} are
sufficient for manifolds to be $\mathcal{Z}$-compactifiable, cf. \cite[Question 2]{Gu20} and \cite[Section 5]{GT03}.}, which is a generalization of completion intended for preserving the homotopy type of the original space.  A compactification
$\widehat{X}=X\sqcup Z$ of a space $X$ is a $\mathcal{Z}$\emph{-compactification} if, for every open set $U\subseteq\widehat{X}$,
$U\backslash Z\hookrightarrow U$ is a homotopy equivalence, where $\sqcup$ denotes a disjoint union. The motivation first came from the modification of the manifold completion applied to Hilbert cube manifolds in \cite{CS76}. This compactification has been proven to be useful in both
geometric group theory and manifold topology, for example, in attacks on the Borel and Novikov Conjectures. An arguably major open problem about $\mathcal{Z}$-compactification is the existence of a characterization of $\mathcal{Z}$-compactifiable manifolds by completely removing Condition (\ref{Char2}) from Theorem \ref{Th: Completion Theorem}, cf. \cite{CS76}, \cite{GT03} and \cite{GG20}.
\begin{question}\label{big question}
Are Conditions (\ref{Char1}), (\ref{Char3}) and (\ref{Char4}) of Theorem \ref{Th: Completion Theorem} sufficient for manifolds to be $\mathcal{Z}$-compactifiable?
\end{question}
In dimensions $\leq 3$, Condition (\ref{Char1}) ensures that $\mathcal{Z}$-compactifications are just completions. Thus, Question \ref{big question} is aimed at $\mathcal{Z}$-compactifications of high-dimensional manifolds.  Although it's unknown whether these conditions are sufficient, in \cite[Thm. 1.2]{GG20}, Guilbault and the author provided a best possible result saying that
an $m$-manifold $M^m$ ($m\geq 5$) satisfies Conditions (\ref{Char1}), (\ref{Char3}) and (\ref{Char4}) of Theorem \ref{Th: Completion Theorem}, if and only if $M^m \times [0,1]$ admits a $\mathcal{Z}$-compactification.

Clearly, completable manifolds are both pseudo-collarable and $\mathcal{Z}$-compactifiable. It is known that various manifolds such as Hilbert cube manifolds (satisfying Conditions (\ref{Char1}), (\ref{Char3}) and (\ref{Char4})) and Davis's manifolds are both pseudo-collarable and $\mathcal{Z}$-compactifiable (see \cite{Fis03} or \cite{AS85}) but not completable. Nevertheless, the relationship between pseudo-collarable manifolds and $\mathcal{Z}$-compactifiable manifolds are not well-understood. Here we answer the following open question in negative.

\begin{question}\label{Question: Z-compact implies pseudo}\cite{Gu20}
Are $\mathcal{Z}$-compactifiable manifolds pseudo-collarable?
\end{question}

\begin{remark}
It is worth noting that Question \ref{Question: Z-compact implies pseudo} is related to an older open question posed in \cite{GT03} asking if
a $\mathcal{Z}$-compactifiable open manifold can fail to be pseudo-collarable. However, Question \ref{Question: Z-compact implies pseudo} is formulated in a more general setting with the manifolds with non-empty boundary taken into account, which is essential to our examples.
\end{remark}
The following theorem is the main result of this paper.
\begin{theorem}\label{Th: Z-compact not pseudo}
There exist $\mathcal{Z}$-compactifiable manifold $M^m$ $(m\geq 4)$ with noncompact boundary which fail to be pseudo-collarable.
\end{theorem}

About the organization of this paper: In Sections \ref{Section: Terminology}-\ref{Section: Peripheral stability}, we recapitulate definitions
and results from earlier papers \cite{GG20} and \cite{Gu20} which
provide preliminaries about neighborhoods of infinity, ends and peripheral perfect semistability condition. The main content is \S \ref{Section: Main content}, where we prove Theorem \ref{Th: Z-compact not pseudo}. The starting point of our 4-dimensional examples is based on the construction of a contractible open 3-manifold $W^3$ which cannot be embedded in any compact 3-manifolds, see  \cite{Gu21} and \cite{KM62}. After introducing $W^3$, we use the interaction of knot theory and combinatorial group theory to show that $W^3 \times [0,1)$ is a desired example. More specifically, in \S \ref{subsection: Z-compactification of W x [0,1)}, we prove that the space $W^3 \times [0,1)$ is $\mathcal{Z}$-compactifiable. In \S \ref{subsection: pi_1 at infinity of W}, we show that $W^3 \times [0,1)$ does not satisfy Condition (\ref{condition b}) of Theorem \ref{Th: Characterization Theorem}. Finally, in \S \ref{subsection: High-dimensional examples}, we produce examples of dimension greater than four, using the structure of $W^3 \times [0,1)$.

\section{Conventions and notations \label{Section: Terminology}}
Manifolds are assumed to be PL. Equivalent results in the DIFF and TOP categories may be obtained in the usual ways. Unless stated
otherwise, an $m$-manifold $M^{m}$ is permitted to have a boundary, denoted
$\partial M^{m}$. We denote the \emph{manifold interior} by
$\operatorname*{int}M^{m}$. For $A\subseteq M^{m}$, the \emph{point-set
interior} will be denoted $\operatorname*{Int}_{M^{m}}A$ and the
\emph{frontier} by $\operatorname{Fr}_{M^{m}}A$. A \emph{closed manifold} is a compact
boundaryless manifold, while an \emph{open manifold} is a non-compact
boundaryless manifold.

A $q$-dimensional submanifold $Q^{q}\subseteq M^{m}$ is
\emph{properly embedded} if it is a closed subset of $M^{m}$ and $Q^{q}%
\cap\partial M^{m}=\partial Q^{q}$; it is \emph{locally flat} if each
$p\in\operatorname*{int}Q^{q}$ has a neighborhood pair homeomorphic to
$\left(
\mathbb{R}
^{m},%
\mathbb{R}
^{q}\right)  $ and each $p\in\partial Q^{q}$ has a neighborhood pair
homeomorphic to $\left(
\mathbb{R}
_{+}^{m},%
\mathbb{R}
_{+}^{q}\right)  $. By this definition, the only properly embedded codimension
0 submanifolds of $M^{m}$ are unions of its connected components; a more
useful variety of codimension 0 submanifolds are the following: a codimension
0 submanifold $Q^{m}\subseteq M^{m}$ is \emph{clean} if it is a closed subset
of $M^{m}$ and $\operatorname{Fr}_{M}Q^{m}$ is a properly embedded locally flat
(hence, bicollared) $\left(  m-1\right)  $-submanifold of $M^{m}$. In that
case, $\overline{M^{m}\backslash Q^{m}}$ is also clean, and $\operatorname{Fr}_{M^m}Q^{m}$ is a clean 
codimension 0 submanifold of both $\partial Q^{m}$ and $\partial(\overline{M^{m}\backslash Q^{m}})$.

When the dimension of a manifold or submanifold is clear, we will often omit
the superscript; for example, denoting a clean codimension 0 submanifold
simply by $Q$. Similarly, when the ambient space is clear, we denote
(point-set) interiors and frontiers by $\operatorname*{Int}A$ and $\operatorname{Fr} A$.

For any codimension 0 clean submanifold $Q\subseteq M^{m}$, let $\partial
_{M}Q$ denote $Q\cap\partial M^{m}$ and $\operatorname*{int}_{M}%
Q=Q\cap\operatorname*{int}M^{m}$; alternatively, $\partial_{M}Q=\partial
Q\backslash\operatorname{int}(\operatorname{Fr} Q)$ and $\operatorname*{int}_{M}Q=Q\backslash\partial M^{m}$.
Note that $\operatorname*{int}_{M}Q$ is a $m$-manifold and $\partial\left(
\operatorname*{int}_{M}Q\right)  =\operatorname*{int}\left( \operatorname{Fr} Q\right)
$.

A metric space $Y$ is said to be an \emph{absolute neighborhood retract} (abbreviated as ANR) if, for
each closed subset $A$ of a metric space $X$, every map $f: A \to Y$ has a continuous extension $F: U \to Y$
defined on some neighborhood $U$ of $A$ in $X$. Every topological manifold is an ANR \cite[III: 8.3]{Hu65}.

\section{Ends, $\operatorname*{pro}$-$\pi_{1}$, and the peripherally perfectly semistable
condition}\label{Section: Peripheral stability}

We recall the definition of the peripherally perfectly semistability developed in \cite{Gu20}, which are essential for this paper. Since all the examples we shall construct in \S \ref{Section: Main content} satisfy the inward tameness (Condition (\ref{condition a})) and the Wall finiteness obstruction (Condition (\ref{condition c})), we omit the illustrations of those conditions. Readers are referred to \cite{Gu20} or \cite{GG20} for details.
\subsection{Neighborhoods of infinity, partial neighborhoods of infinity, and ends.} Let $M^{m}$ be a connected manifold. A \emph{clean neighborhood of infinity}
in $M^{m}$ is a clean codimension 0 submanifold $N\subseteq M^{m}$ for which
$\overline{M^{m}\backslash N}$ is compact. Equivalently, a clean neighborhood
of infinity is a set of the form $\overline{M^{m}\backslash C}$ where $C$ is a
compact clean codimension 0 submanifold of $M^{m}$. A \emph{clean compact exhaustion} of
$M^{m}$ is a sequence $\left\{  C_{i}\right\}  _{i=1}^{\infty}$ of clean
compact connected codimension 0 submanifolds with $C_{i}\subseteq
\operatorname*{Int}_{M^{m}}C_{i+1}$ and $\cup C_{i}=M^{m}$. By letting
$N_{i}=\overline{M^{m}\backslash C_{i}}$ we obtain the corresponding
\emph{cofinal sequence of clean neighborhoods of infinity}. Each such $N_{i}$
has finitely many components $\left\{  N_{i}^{j}\right\}  _{j=1}^{k_{i}}$. By
enlarging $C_{i}$ to include all of the compact components of $N_{i}$ we can
arrange that each $N_{i}^{j}$ is noncompact; then, by drilling out regular
neighborhoods of arcs connecting the various components of each $\operatorname{Fr}
_{M^{m}}N_{i}^{j}$ (thereby further enlarginging $C_{i}$), we can
arrange that each $\operatorname{Fr}_{M^{m}}N_{i}^{j}$ is connected. An $N_{i}$ with
these latter two properties is called a $0$-neighborhood of infinity. For
convenience, most constructions in this paper will begin with a clean compact
exhaustion of $M^{m}$ with a corresponding cofinal sequence of clean
0-neighborhoods of infinity.

Assuming the above arrangement, we let $N_{1}^{j,1}=N_{1}^{j}$ and $N_{i+1}^{j,i+1}= N_{i+1}^{l}\subseteq N_{i}^{j,i}$, where $l \in \{1,\dots, k_{i+1}\}$. An \emph{end }of $M^{m}$ is determined by a
nested sequence of components $\varepsilon=\left(  N_{i}^{j,i}\right)
_{i=1}^{\infty}$ of the $N_{i}$; each component is called a \emph{neighborhood
of }$\varepsilon$. In $\S 3.3$, we discuss components $\{N^j\}$ of a neighborhood of infinity
$N$ without reference to a specific end of $M^m$. In that situation, we will refer to the
$N^j$ as \emph{partial neighborhoods of infinity} for $M^m$ (\emph{partial $0$-neighborhoods} if $N$ is a
0-neighborhood of infinity). Clearly every noncompact clean connected codimension
0 submanifold of $M^m$ with compact frontier is a partial neighborhood of infinity with respect to an appropriately chosen compact $C$; if its frontier is connected it is a partial 0-neighborhood of infinity.

\subsection{The fundamental group of an end} For each end $\varepsilon$, we will define the \emph{fundamental group at}
$\varepsilon$; this is done using inverse sequences. Two inverse sequences of
groups and homomorphisms $A_{0}\overset{\alpha_{1}}{\longleftarrow}%
A_{1}\overset{\alpha_{2}}{\longleftarrow}A_{3}\overset{\alpha_{3}%
}{\longleftarrow}\cdots$ and $B_{0}\overset{\beta_{1}}{\longleftarrow}%
B_{1}\overset{_{\beta_{2}}}{\longleftarrow}B_{3}\overset{\beta_{3}%
}{\longleftarrow}\cdots$ are \emph{pro-isomorphic} if they contain
subsequences that fit into a commutative diagram of the form
\begin{equation}
\begin{diagram} A_{i_{0}} & & \lTo^{\lambda_{i_{0}+1,i_{1}}} & & A_{i_{1}} & & \lTo^{\lambda_{i_{1}+1,i_{2}}} & & A_{i_{2}} & & \lTo^{\lambda_{i_{2}+1,i_{3}}}& & A_{i_{3}}& \cdots\\ & \luTo & & \ldTo & & \luTo & & \ldTo & & \luTo & & \ldTo &\\ & & B_{j_{0}} & & \lTo^{\mu_{j_{0}+1,j_{1}}} & & B_{j_{1}} & & \lTo^{\mu_{j_{1}+1,j_{2}}}& & B_{j_{2}} & & \lTo^{\mu_{j_{2}+1,j_{3}}} & & \cdots \end{diagram} \label{basic ladder diagram}%
\end{equation}
An inverse sequence is \emph{stable }if it is pro-isomorphic to a constant
sequence $C\overset{\operatorname*{id}}{\longleftarrow}%
C\overset{\operatorname*{id}}{\longleftarrow}C\overset{\operatorname*{id}%
}{\longleftarrow}\cdots$. Clearly, an inverse sequence is pro-isomorphic to
each of its subsequences; it is stable if and only if it contains a
subsequence for which the images stabilize in the following manner
\begin{equation}
\begin{diagram} A_{0}& & \lTo^{{\lambda}_{1}} & & A_{1} & & \lTo^{{\lambda}_{2}} & & A_{2} & & \lTo^{{\lambda}_{3}} & & A_{3} &\cdots\\ & \luTo & & \ldTo & & \luTo & & \ldTo & & \luTo & & \ldTo & \\ & & \operatorname{Im}\left( \lambda_{1}\right) & & \lTo^{\cong} & & \operatorname{Im}\left( \lambda _{2}\right) & &\lTo^{\cong} & & \operatorname{Im}\left( \lambda_{3}\right) & & \lTo^{\cong} & &\cdots & \\ \end{diagram} \label{Standard stability ladder}%
\end{equation}
where all unlabeled homomorphisms are restrictions or inclusions.

Given an end $\varepsilon=\left(  N_{i}^{k_{i}}\right)  _{i=1}^{\infty}$,
choose a ray $r:[1,\infty)\rightarrow M^{m}$ such that $r\left(
[i,\infty)\right)  \subseteq N_{i}^{k_{i}}$ for each integer $i>0$ and form
the inverse sequence
\begin{equation}
\pi_{1}\left(  N_{1}^{k_{1}},r\left(  1\right)  \right)  \overset{\lambda
_{2}}{\longleftarrow}\pi_{1}\left(  N_{2}^{k_{2}},r\left(  2\right)  \right)
\overset{\lambda_{3}}{\longleftarrow}\pi_{1}\left(  N_{3}^{k_{3}},r\left(
3\right)  \right)  \overset{\lambda_{4}}{\longleftarrow}\cdots
\label{sequence: pro-pi1}%
\end{equation}
where each $\lambda_{i}$ is an inclusion induced homomorphism composed with
the change-of-basepoint isomorphism induced by the path $\left.  r\right\vert
_{\left[  i-1,i\right]  }$. We refer to $r$ as the \emph{base ray} and the
sequence (\ref{sequence: pro-pi1}) as a representative of the
\textquotedblleft fundamental group at $\varepsilon$ based at $r$%
\textquotedblright\ ---denoted $\operatorname*{pro}$-$\pi_{1}\left(
\varepsilon,r\right)  $. We say \emph{the fundamental
group at }$\varepsilon$\emph{ is stable} if (\ref{sequence: pro-pi1}) is a
stable sequence. 

If $\{H_i,\mu_i\}$ can be chosen so that each $\mu_i$ is an epimorphism,
we say that our inverse sequence is \emph{semistable} (or \emph{Mittag-Leffler}, or \emph{pro-epimorphic}).
In this case, it can be arranged that the restriction maps in the bottom row of (\ref{basic ladder diagram}) are
epimorphisms. Similarly, if $\{H_i,\mu_i\}$ can be chosen so that each $\mu_i$ is a monomorphism,
we say that our inverse sequence is \emph{pro-monomorphic}; it can then be arranged that the
restriction maps in the bottom row of (\ref{basic ladder diagram}) are monomorphisms. It is easy to see that an
inverse sequence that is semistable and pro-monomorphic is stable.  A nontrivial (but standard) observation is that both semistability and stability of $\varepsilon$ do
not depend on the base ray (or the system of neighborhoods if infinity used to
define it). See \cite{Gui16} or \cite{Geo08}.

Recall that a \emph{commutator} element of a group $H$ is an element of the form $x^{-1}y^{-1}xy$
where $x,y \in H$; and the \emph{commutator subgroup} of $H$; denoted $[H,H]$ or $H^{(1)}$, is the subgroup generated by all of its commutators. The group $H$ is \emph{perfect} if $H = [H,H]$. An inverse sequence of groups is \emph{perfectly semistable} if it is pro-isomorphic to an inverse
sequence.

\begin{equation}
G_0 \xtwoheadleftarrow{\lambda_1} G_1 \xtwoheadleftarrow{\lambda_2}G_2 \xtwoheadleftarrow{\lambda_3}\cdots 
\end{equation}
of finitely generated groups and surjections where each $\ker(\lambda_i)$ perfect. 

\subsection{Relative connectedness, relatively perfect semistability, and the peripheral perfect semistability condition.}

Let $Q$ be a manifold and $A\subseteq\partial Q$. We say that $Q$ is\emph{
}$A$-\emph{connected at infinity} if $Q$
contains arbitrarily small neighborhoods of infinity $V$ for which $A\cup V$
is connected.

If $A\subseteq\partial Q$ and $Q$ is $A$-connected at infinity: let $\left\{  V_{i}\right\}  $ be a cofinal sequence
of clean neighborhoods of infinity for which each $A\cup V_{i}$ is connected;
choose a ray $r:[1,\infty)\rightarrow\operatorname*{Int}Q$ such that $r\left(
[i,\infty)\right)  \subseteq V_{i}$ for each $i>0$; and form the inverse
sequence%
\begin{equation}
\pi_{1}\left(  A\cup V_{1},r\left(  1\right)  \right)  \overset{\mu
_{2}}{\longleftarrow}\pi_{1}\left(  A\cup V_{2},r\left(  2\right)  \right)
\overset{\mu_{3}}{\longleftarrow}\pi_{1}\left(  A\cup V_{3},r\left(  3\right)
\right)  \overset{\mu_{4}}{\longleftarrow}\cdots
\label{sequence: rel A pro-pi1}%
\end{equation}
where bonding homomorphisms are obtained as in (\ref{sequence: pro-pi1}). We
say $Q$ is $A$-\emph{perfectly $\pi_1$-semistable at infinity } (resp. $A$-$\pi_{1}%
$\emph{-stable at infinity}) if (\ref{sequence: rel A pro-pi1}) is perfectly semistable (resp. stable).
Independence of this property from the choices of $\left\{  V_{i}\right\}  $
and $r$ follows from the traditional theory of ends by applying \cite[Lemmas 4.1 \& 4.2]{GG20}. 

\begin{definition} Let $M^m$ be an manifold and $\varepsilon$ be an end of $M^m$
\begin{enumerate}
\item \label{condition1} $M^m$ is \emph{peripherally locally connected at infinity} if it contains arbitrarily small
$0$-neighborhoods of infinity $N$ with the property that each component $N^j$ is $\partial_M N^j$-connected at infinity.
\item $M^m$ \label{condition2} is \emph{peripherally locally connected at} $\varepsilon$ if $\varepsilon$ has arbitrarily small $0$-neighbor-hoods $P$ that are $\partial_M P$-connected at infinity.
\end{enumerate} 
\end{definition}

An $N$ with the property described in Condition (\ref{condition1}) above will be called a \emph{strong $0$-neighborhood of infinity} for $M^m$, and a $P$ with the property described in Condition (\ref{condition2})
will be called a \emph{strong $0$-neighborhood of $\varepsilon$}. More generally, any connected partial
0-neighborhood of infinity $Q$ that is $\partial_M Q$-connected at infinity will be called a \emph{strong
partial $0$-neighborhood of infinity}.

Since every inward tame manifold $M^{m}$ is peripherally locally connected at infinity \cite[Cor. 5.5]{GG20}, that condition plays less prominent role than the next definition.

\begin{definition}Let $M^m$ be a manifold and $\varepsilon$ an end of $M^m$.
\begin{enumerate}
\item $M^m$ is \emph{peripherally perfectly $\pi_1$-semistable at infinity} if it contains arbitrarily small strong
0-neighborhoods of infinity $N$ with the property that each component $N^j$ is $\partial_M N^j$-perfectly $\pi_1$-semistable at infinity.

\item $M^m$ is \emph{peripherally perfectly $\pi_1$-semistable at $\varepsilon$} if $\varepsilon$ has arbitrarily small strong $0$-neighborhoods $P$ that are $\partial_M P$-perfectly $\pi_1$-semistable at infinity.
\end{enumerate}
 \end{definition}

\begin{remark}
If $M^m$ contains arbitrarily small 0-neighborhoods of infinity $N$ with
the property that each component $N^j$ is $\partial_M N^j$-perfectly semistable at infinity, then those components provide arbitrarily small neighborhoods of the ends satisfying the necessary perfectly semistable condition. Thus, it's easy to see peripheral perfect $\pi_1$-semistability at infinity implies peripheral perfect $\pi_1$-semistability at each end.
\end{remark}

A locally finite polyhedron $P$ is \emph{inward tame} if it contains
arbitrarily small polyhedral neighborhoods of infinity that are finitely
dominated. 


\begin{lemma}\label{Lemma: inward tameness implies semistable}
Let $M^{m}$ be inward tame and peripherally locally connected at infinity. Then every strong
partial $0$-neighborhood of infinity $Q\subseteq M^{m}$ is $\partial_M Q$-$\pi_1$-semistable.
\end{lemma}
\begin{proof}
By \cite[Lemma 5.3]{GG20}, $Q\backslash \partial M^m$ is inward tame. Using \cite[Lemma 4.2]{GG20}, proving $Q\subseteq M^{m}$ is $\partial_M Q$-$\pi_1$-semistable is 
equivalent to showing that $Q\backslash \partial M^m$ has semistable pro-$\pi_1$. The latter follows from a slight modification
of \cite[Prop. 3.2]{GT03} (or a fact that "degree 1 maps between manifolds induce
surjections on fundamental groups" as suggested in \cite[Rmk. 4]{GT03}).
\end{proof}

\section{$\mathcal{Z}$-compactifiable but not pseudo-collarable manifolds with boundary}\label{Section: Main content}
We shall prove Theorem \ref{Th: Z-compact not pseudo} by focusing on the fundamental group at infinity of certain contractible manifolds with noncompact boundary.

\subsection{$\mathcal{Z}$-compactification of $W^3 \times [0,1)$}\label{subsection: Z-compactification of W x [0,1)}
Before constructing the contractible open $3$-manifold $W^3$, we first give a general result to guarantee that the product of $W^3$ with $[0,1)$ is $\mathcal{Z}$-compactifiable.

\begin{proposition} \label{Prop: contractible is Z-compact}
Let $M$ be a contractible open manifold. Then $M \times [0,1)$ is $\mathcal{Z}$-compactifiable.
\end{proposition}

The proof of Proposition \ref{Prop: contractible is Z-compact} is based on the next result.

\begin{lemma}\cite[Prop. 2.1]{BM91}\label{Lemma: Z compactification criterion}
\label{Prop: one-point compact}
Suppose that $X$ is a compactum and $Z \subset X$ a closed subset such that
\begin{enumerate}[(i)]
\item \label{condition empty set} $\operatorname{Int} Z = \emptyset,$
\item \label{condition dim}     $\dim X = n < \infty,$
\item \label{condition trivial loop}  for every $k = 0, 1, \dots, n$, every point $z \in Z$ and every neighborhood $U$ of $z$, there exists a neighborhood $V$ of $z$ such that $\alpha: S^k \to V \backslash Z$ extends to $\tilde{\alpha}: B^{k+1} \to U \backslash Z$ and
\item \label{condition ANR}    $X \backslash Z$ is an ANR.
\end{enumerate}
Then $X$ is an ANR and $Z \subset X$ is a $\mathcal{Z}$-set.
\end{lemma}

\begin{proof}[Proof of Proposition \ref{Prop: contractible is Z-compact}]
The strategy is to show the one-point compactification $M \times [0,1) \sqcup \{\infty\}$ is a $\mathcal{Z}$-compactification. To that end, we verify Conditions (\ref{condition empty set})-(\ref{condition ANR}) in Lemma \ref{Lemma: Z compactification criterion}.  Verifying Conditions (\ref{condition empty set}) and (\ref{condition dim}) are trivial. The standard ANR theory shows that $M \times [0,1)$ is an ANR. To check Condition (\ref{condition trivial loop}), it suffices to show that for each neighborhood $U$ of infinity we may choose some neighborhood $V$ of infinity in $U$ which contracts in itself. Let us pick a compactum $C\subset M$ and a number $a\in (0,1)$ such that $K = C \times [0,a]$ containing $\overline{M\times [0,1)\backslash U}$.  Then $V =  M\times [0,1) \backslash K = (M \backslash C) \times [0,1) \cup M \times [a,1)$ is the desired neighborhood of infinity. To see that, we observe that there exists an obvious deformation retract of $V$ onto $M \times \{b\}$, where $b\in [a,1)$. Since $M \times \{b\}$ deformation retracts to a point in itself, the composition of the above two deformation retractions provides a contraction of $V$ in itself.
\end{proof}

\subsection{The fundamental group at infinity of $W^3 \times [0,1)$}\label{subsection: pi_1 at infinity of W}

The result below and the necessity of Theorem \ref{Th: Characterization Theorem} imply that the only way to see that $W^3 \times [0,1)$ is not pseudo-collarable is to show that the peripherally perfect semistability at infinity fails.

\begin{proposition}\cite[Prop. 12.1]{GG20}
Let $M$ be a connected open manifold. If $M$ has finite homotopy type, then $M \times [0, 1)$ is one-ended and inward tame, with $\sigma_\infty(M \times [0,1))=0$.
\end{proposition}

First, we prepare some algebraic devices.

\begin{definition}
A group $G$ is said to be \emph{hypoabelian} if the following equivalent conditions are satisfied:
\begin{enumerate}
\item $G$ contains no nontrivial perfect subgroup.
\item The transfinite derived series\footnote{The transfinite derived series of a group is an extension of its derived series such that the successor of a given subgroup is its commutator subgroup, and subgroups at limit ordinals are given by intersecting all previous subgroups.} terminates at the identity.
\end{enumerate}
\end{definition}

\begin{example}
\begin{enumerate}

\item Every abelian group is hypoabelian.

\item Solvable groups, residually solvable groups and free groups are hypoabelian.

\item Every right-angled Artin group is hypoabelian.

\item The Baumslag-Solitar groups $\operatorname{BS}(1,n)$ are solvable, thus, hypoabelian.

\end{enumerate}
\end{example}

Hypoabelianity is closed under a number of operations:

\begin{enumerate}

\item Free products of hypoabelian groups are hypoabelian.

\item Every extension of a hypoabelian group by a hypoabelian group is hypoabelian.

\item Let $H$ be a subgroup of a group $G$. Since conjugate subgroups are isomorphic, if $H$ is hypoabelian, so are the conjugates of $H$.
\end{enumerate}


\begin{definition}
A \emph{split} HNN-\emph{extension} $G$ of a group $B$ is an HNN-extension $G = \langle B,t | t^{-1}\eta(a)t = \iota(a) \rangle$, where $a\in A$ such that one of the injections $\eta,\iota: A \hookrightarrow B$ (say $\eta$) \emph{splits}, that is, $B$ is a split extension $N \lhook\joinrel\longrightarrow B  \overset{\xtwoheadrightarrow{\phi}}{\underset{\eta}{\longleftarrow\joinrel\rhook}} A$.

\end{definition}

\begin{definition}
A group $G$ is a \emph{split amalgamated free product} iff $G$ may be expressed as an amalgamated free product $B\ast_A C$, where one of the injections $A \hookrightarrow B$ or $A \hookrightarrow C$ splits.
\end{definition}

The results below show that hypoabelianism can be preserved under delicate group operations.

\begin{proposition}\cite[Thm. E]{How79}
\label{Prop: splitting is hypoabelian}
Split amalgamated free products of hypoabelian groups are hypoabelian.
\end{proposition}

Given a nonempty set of groups $\{G_\lambda|\lambda \in \Lambda\}$ together with a group $H$ which
is isomorphic with a subgroup $H_\lambda$ of $G_\lambda$ by means of monomorphism $\phi_\lambda: H \to G_\lambda$. There is a frequently used term known as the free product of the $G_\lambda$'s with the amalgamated subgroup $H$. Roughly speaking, this is the largest
group generated by the $G_\lambda$'s in which the subgroups $H_\lambda$ are identified with $H$ by $\phi_\lambda$. Such groups are known as 
\emph{generalized free products}. Readers are referred to \cite{Rob95} for a more explicit description.

Let $A$ be a hypoabelian group. Then there exists a transfinite derived series of $A$ terminates at some ordinal $\alpha$. We call all derived subgroups a \emph{filtration} of $A$. Let $B$ be a second hypoabelian group with a transfinite derived series terminates at some ordinal $\beta$, and let $a\in A$, $b\in B$ be two non-identity elements. There exist ordinals $m\in \{0,1,\dots,\alpha\}$ and $n\in \{0,1,\dots,\beta\}$ such that the generators $a \in A^{(m)} \backslash  A^{(m+1)}$ and
$b \in B^{(n)} \backslash B^{(n+1)}$, where $A^{(m)}$ and $B^{(n)}$ denote the $m$-th
and $n$-th derived subgroups of $A$ and $B$, respectively. Then we say the filtrations of $A$ and $B$ are \emph{compatible} if cosets $aA^{(m+1)}$ and
$bB^{(n+1)}$ have the same order in factor groups $A/A^{(m+1)}$ and $B/B^{(n+1)}$, respectively.

\begin{proposition}\label{Prop: free product of hypoabelian}
The generalized free product of two hypoabelian groups $A$ and $B$ amalgamating a cyclic subgroup is hypoabelian, provided that the filtrations of $A$ and $B$ are compatible.
\end{proposition}
\begin{proof}
The proof uses a similar argument in the proof of \cite[Thm. 2]{KM11}. Let $a \in A$, $b \in B$ be
two non-identity elements. Assuming the notions above, we establish the amalgamated product $G = \{A\ast B; a = b\}$ of $A$ and $B$
 with elements $a$ and $b$ identified. Let
$C$ be the central product of $A/A^{(m+1)}$ and $B/B^{(n+1)}$ amalgamating $aA^{(m+1)}$ with
$bB^{(n+1)}$:
$$C = \{A/A^{(m+1)} \times B/B^{(n+1)};aA^{(m+1)}=bB^{(n+1)}\}.$$
Define a homomorphism $\phi: G \to C$ sending $a_i$ to $(a_i A^{(m+1)}, B^{(n+1)})$
and $b_j$ to $(A^{(m+1)}, b_j B^{(n+1)})$, where $a_i,b_j$ are generators of $A,B$, respectively. Let $K = \ker \phi$ and $D = \langle a \rangle = \langle b \rangle$. Then $K \cap D = 1$.
By \cite{Neu49}, $K$ is actually a free product of conjugates of subgroups of $A$ and $B$, and a free group. Thus, $K$ is
hypoabelian. Since $G$ is an extension of a hypoabelian group by a hypoabelian group, $G$ is also hypoabelian.
\end{proof}
\begin{remark}\label{Remark: compatibility}
Without the compatibility, we may not conclude that the cyclic amalgamation is hypoabelian even though $A$ and $B$ are assumed to be solvable or nilpotent. The proofs of \cite[Thm. 2]{KM11} and \cite[Thm. 2]{Kah11} are flawed since both the arguments rely on the compatibility in the central product structure. It follows from \cite[Thm. A]{GS99} that there are torsion-free metabelian groups whose commutator quotients are nontrivial with prime exponent for any odd prime $p$. Choosing groups $A$ and $B$ with different primes $p$, say 3 and 5, gives infinite order elements $a\in A$ and $b\in B$ (not in the commutator subgroup) that have different orders modulo the commutator subgroup. Hence, the central product structure fails.
\end{remark}


The following lemma quoted from \cite[Lemma 4.1]{GT03} indicates the strategy we will use to prove the main theorem.
\begin{lemma} \label{Lemma:hypoabelian groups}
Let 
\begin{equation}
G_0 \leftarrow G_1 \leftarrow G_2 \leftarrow \cdots
\end{equation}
be an inverse sequence of groups with surjective but non-injective bonding homomorphism. Suppose each $G_i$ is a hypoabelian group. Then the inverse sequence is not perfectly semistable.
\end{lemma}

Now, we reproduce a famous example originially proposed by Bing, i.e., a 3-dimensonal contractible open manifold $W^3$,\footnote{Haken \cite{Hak68} first proved that $W^3$ embed in no compact 3-manifold. Recently, the author \cite{Gu21} showed that $W^3$ cannot even embed in a much more general compact space --- a compact, locally connected and locally 1-connected 3-dimensional metric space.} see \cite{KM62}, \cite{Hak68} or \cite{Gu21}.

Let $\{T_l|l = 0,1,2,\dots \}$ be a collection of disjoint solid tori standardly embedded in $S^3$. Let the solid torus $T_l'$ be embedded in $\operatorname{Int}T_l$ as in Figure \ref{3_1knot}. The cube $C_l$ contains a tamely embedded thickened trefoil knot with the interior of a standard ball removed. In this paper, the choice of such a knot is not unique. Alternating (or more generally adequate) knots of nonzero writhe number and whose knot groups are hypoabelian can be applied to create more options for $K$ (look ahead to Corollary \ref{Cor: Infinitely many 3D examples}). To produce more examples, one simply removes a standard ball from a thickened knot $K$ and use the resulting space to replace the trefoil-knotted part in $C_l$. Let the oriented simple closed curves $\alpha_l$, $\beta_l$, $\gamma_l$ and $\delta_l$ be as shown in Figure \ref{3_1knot}. The curves $\alpha_l$ and $\beta_l$ are transverse in $\partial T_l$, and meet at the point $q_l \in \partial T_l$. In a similar fashion, the curves $\gamma_l$ and $\delta_l$ are transverse in $\partial T_l'$, and meet at the point $p_l \in \partial T_l'$. For $l \geq 1$, let $L_l = T_l \backslash \operatorname{Int}T_l'$. Define an embedding $h_{l+1}^{l}: T_l \to T_{l+1}$ so that $T_l$ is carried onto $T_{l+1}'$ with $h_{l+1}^{l}(\alpha_l) = \delta_{l+1}$ and $h_{l+1}^{l}(\beta_l) = \gamma_{l+1}$. $W^3$ is the direct limit of the $T_l$'s and denoted as $W^3 = \lim\limits_{l\to \infty}(T_l,h_{l+1}^{l})$. That is equivalent to viewing $W^3$ as the quotient space: $\sqcup_l T_l \xrightarrow{q} W^3$, where $\sqcup_l T_l$ is the disjoint union of the $T_l$'s and $q$ is the quotient map induced by the relation $\sim$ on $\sqcup_l T_l$. If $x\in T_i$ and $y\in T_j$, then $x \sim y$ iff there exists a $k$ larger than $i$ and $j$ such that $h_{k}^{i}(x) = h_{k}^{j}(y)$, where $h_{t}^{s} = h_{t}^{t-1} \circ h_{t-1}^{t-2} \circ \cdots \circ h_{s+2}^{s+1} \circ h_{s+1}^{s}$ for $t > s$. Let $\iota_l: T_l \hookrightarrow \sqcup_l T_l$ be the obvious inclusion map. The composition $q \circ \iota_l$ embeds $T_l$ in $W^3$ as a closed subset. The injectivity follows from the injectivity of $h_{k+1}^{k}$. It is closed since for $j > l$ the set $h_{j}^{l}(T_l)$ is closed in $T_j$. Let $T_l^*$ denote $q \circ \iota_l(T_l)$. $T_l^*$ is embedded in $T_{l+1}^*$ just as the way $h_{l+1}^{l}(T_l)$ ($= T_{l+1}'$) is embedded in $T_{l+1}$. 
Hence, Figure \ref{3_1knot} can be viewed as a picture of the embedding of $T_l^*$ in $T_{l+1}^{*}$. In general, for $k > l$, $T_l^*$ is embedded in $T_k^*$ just as $h_{k}^{l} (T_l)$ is embedded in $T_k$. 

\begin{figure}[h!]
        \centering
       \includegraphics[ width=8cm, height=10cm]{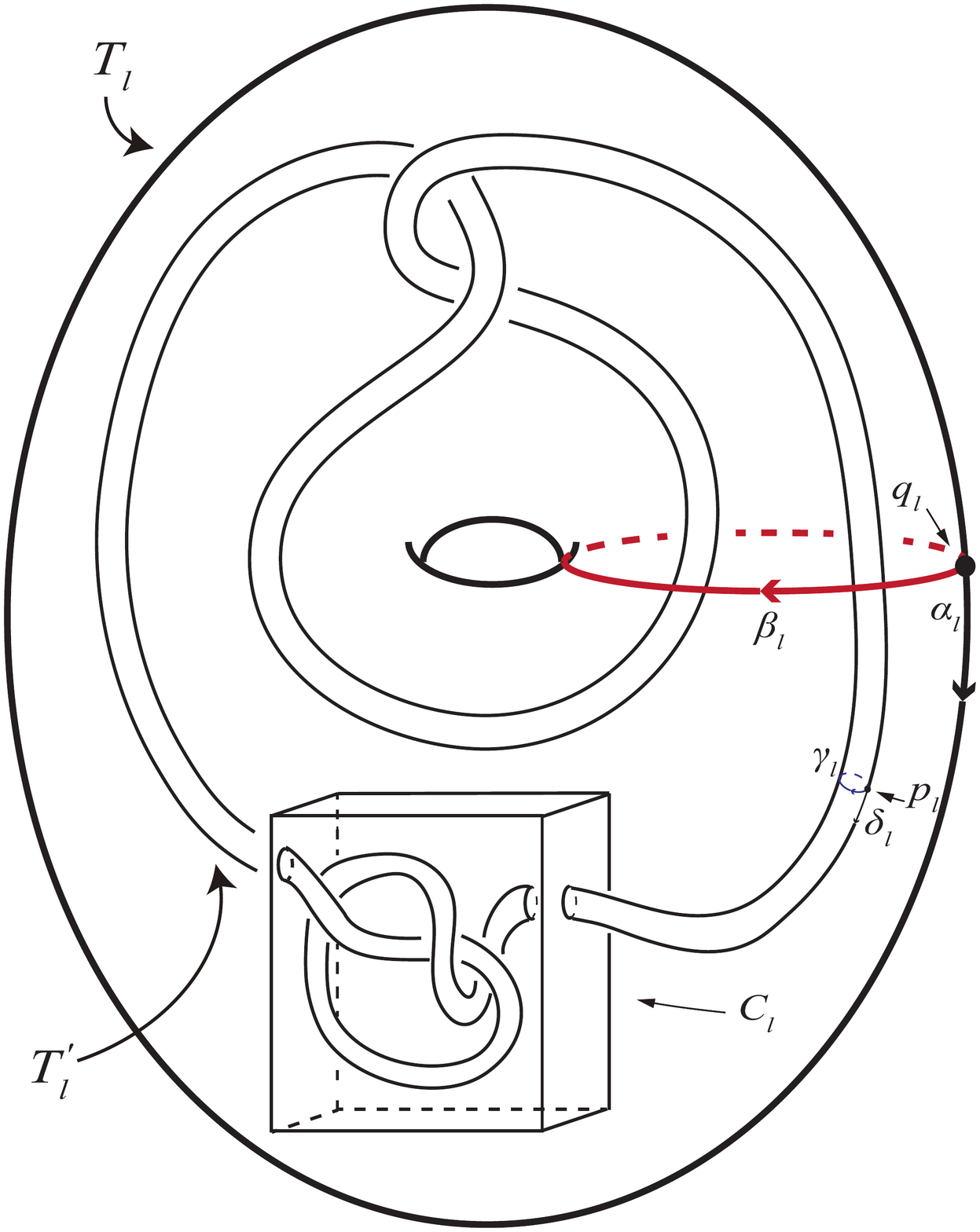}
       \caption{$L_l = T_l \backslash T_l'$. The "inner" boundary component of $L_l$ is $\partial T_l'$. The "outer" boundary component 
       of $L_l$ is $\partial T_l$}
        \label{3_1knot}
\end{figure}

\begin{proposition}
$W^3$ as constructed above is an contractible open connected $3$-manifold.
\end{proposition}
\begin{proof}
See the proof of \cite[Prop. 2.1]{Gu21}.
\end{proof}

To show that the fundamental group at infinity of $W^3 \times [0,1)$ is not perfectly semistable, we start from a specific cofinal sequence of neighborhoods $\{N_j\}_{j=0}^{\infty}$ of infinity of $W^3$, where $N_j = \cup_{l=j}^{\infty} L_{l}^*$ and $L_l^* = T_{l+1}^*\backslash \operatorname{Int} T_{l}^*$.  Taking the product of $N_j$ with $[0,1)$ we define $V_j = N_j \times [0,1) \cup W^3 \times [1-\frac{1}{j+2},1)$. Then $\left\{  V_{j}  \right\}_{j= 1}^{\infty}$ is a cofinal sequence of clean neighborhoods of infinity for which each $N_0 \cup V_{j}$ is connected. Form an inverse sequence in which baserays are suppressed
\begin{equation}
\pi_{1}\left(  N_0 \cup V_{1} \right)  \overset{\mu
_{2}}{\longleftarrow}\pi_{1}\left(  N_0 \cup V_{2}  \right)
\overset{\mu_{3}}{\longleftarrow}\pi_{1}\left(  N_0\cup V_{3}\right)  \overset{\mu_{4}}{\longleftarrow}\cdots
\label{pi_1 of W^3}
\end{equation}
Since $V_j$ deformation retracts onto $W^3 \times \{1-\frac{1}{j+2}\}$ and $W^3$ is contractible, $\pi_{1}\left(  N_0 \cup V_{j} \right)\cong \pi_{1}\left(  L_0^* \cup \cdots  \cup L_{j-1}^*/\partial T_j^*\right)$. Thus, Sequence (\ref{pi_1 of W^3}) is pro-isomorphic to the following sequence
with basepoints suppressed
\begin{equation}
\pi_{1}\left(  L_0^*/\partial T_1^* \right)  \overset{\mu
_{2}}{\longleftarrow}\pi_{1}\left(  L_0^* \cup L_1^* /\partial T_2^* \right)
\overset{\mu_{3}}{\longleftarrow}\pi_{1}\left( L_0^* \cup L_1^* \cup L_2^*/\partial T_3^* \right)  \overset{\mu_{4}}{\longleftarrow}\cdots
\label{pi_1 of quotient space of W^3}
\end{equation}
To better understand each term in Sequence (\ref{pi_1 of quotient space of W^3}), we consider the knot complement $K_j$ defined as $K_j = S^3 \backslash \operatorname{Int}h_{j}^{0} (T_0)$ for $j \geq 1$. That is, $K_j$ is obtained by sewing the solid torus $S^3 \backslash \operatorname{Int}T_j$ to $T_j \backslash \operatorname{Int}h_{j}^0 (T_0)$ along $\partial T_j$. By the construction of $W^3$, the pair $(T_j^*, T_0^*)$ is homeomorphic to $(T_j, h_j^0 (T_0))$.
Thus, 
$$\pi_1 \left( (T_j \backslash \operatorname{Int} h_j^0 (T_0)) /\partial T_j  \right) \cong \pi_{1}\left( L_0^* \cup \cdots  \cup L_{j-1}^*/\partial T_j^* \right).$$ 
By \cite[Claim 2]{Gu21} knot group $\pi_1(K_j)$ is isomorphic to $\pi_1 \left( (T_j \backslash \operatorname{Int} h_j^0 (T_0)) /\partial T_j  \right)$.  Hence, Sequence (\ref{pi_1 of quotient space of W^3}) is pro-isomorphic to a sequence of knot groups

\begin{equation}
\pi_{1}\left(  K_1 \right)  \overset{\mu
_{2}}{\longleftarrow}\pi_{1}\left(  K_2 \right)
\overset{\mu_{3}}{\longleftarrow}\pi_{1}\left( K_3 \right)  \overset{\mu_{4}}{\longleftarrow}\cdots
\label{knot groups}
\end{equation}
The details about the construction of knot space $K_j$ can be found in \cite{Gu21}. It's clear that $\pi_1(K_1)$ is isomorphic to $\pi_1(S^3 \backslash K)$. As one option, we can choose a fibered knot for $K$. Recall that a knot $k \subset S^3$ is \emph{fibered} if its exterior admits a locally trivial fibration over $S^1$.

\begin{proposition}\label{Prop: 9_42 is hypoabelian}
Fibered knot groups are hypoabelian.
\end{proposition}
\begin{proof}
By \cite{Neu65}, the commutator subgroup of the fibered knot group is finitely generated, i.e., free. 
\end{proof}
\begin{remark}
The Alexander polynomial of fibered knot is monic \cite{Neu65} \cite{Rap60} \cite{Sta61}, i.e., the coefficient of the highest degree term of the normalized Alexander polynomial is a unit $\pm 1$. By the duality of the Alexander polynomial, its lowest degree term is also $\pm 1$. This criterion is sufficient for alternating knots \cite{Mur63} and prime knots up to 10 crossings \cite{Kan79}. In general, the converse is not true. 
The Alexander polynomial of any fibered knot is also the Alexander polynomial of infinitely many nonfibered knots.
\end{remark}

%
%

Next, we briefly discuss the concepts of crossing number, writhe and linking number and the connections among them. \emph{Crossing number} is the minimal number
of simple self-intersections which appear in a knot diagram or link diagram, i.e., a planar picture (regular projection) of a knot or link, of a given type.
The \emph{signs} of crossings in a knot diagram (equipped with an orientation) are defined by the (right-hand) rules in Figure \ref{Signs of crossings in a knot diagram}.
\begin{figure}[h!]
\centering
\begin{subfigure}{0.4\textwidth}
\centering
\begin{tikzpicture}
\node (A) at (0,1) {};
\node (B) at (1,-1) {};
\node (C) at (1,1) {};
\node (D) at (0,-1) {};
\node (i) at (intersection of A--B and D--C) {}; 
\path[->] (D) edge[thick] (C);
\path[-] (B) edge[thick] (i);
\path[->] (i) edge[thick] (A);
\end{tikzpicture}\caption{Positive crossing.}
\end{subfigure}
\begin{subfigure}{0.4\textwidth}
\centering
\begin{tikzpicture}
\centering
\node (A) at (0,1) {};
\node (B) at (1,-1) {};
\node (C) at (1,1) {};
\node (D) at (0,-1) {};
\node (i) at (intersection of A--B and D--C) {}; 
\path[->] (B) edge[thick] (A);
\path[-] (D) edge[thick] (i);
\path[->] (i) edge[thick] (C);
\end{tikzpicture}\caption{Negative crossing.}
\end{subfigure}
\caption{Signs of crossings in a knot diagram.}
\label{Signs of crossings in a knot diagram}
\end{figure}
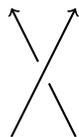
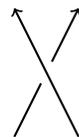
The \emph{writhe} of a knot $k$, denoted by $w(k)$, is the total number of positive crossings minus the total number of negative crossings. The writhe is the same for
either orientation. Suppose we have an oriented link diagram with two components $C_1$ and $C_2$. Let $n_1$ (resp. $n_2$) be the sum of the signs of crossings between a pair of components $C_1$ and $C_2$ of a link in which the over arc belongs to $C_1$ (resp. $C_2$). Then the \emph{linking number}, denoted $\operatorname{lk}(C_1,C_2)$, is equal to $(n_1+n_2)/2$. This is the same labeling used to compute the writhe of a knot, though in this case only crossings that involve both $C_1$ and $C_2$ are counted. The other seven equivalent definitions of the link number can be found in \cite[Section 5D]{Rol76}. Consider a closed ribbon with disjoint boundary components $C_1,C_2$ and the axis of the ribbon $C$. Then the relationship between the linking number and the writhe of $C$ can be expressed by the White formula which was discovered independently by C\v{a}lug\v{a}reanu\cite{Cal61}, Fuller\cite{Ful71}, Pohl\cite{Poh68}, and White\cite{Whi69}:
$$\operatorname{lk}(C_1,C_2) = t(C_1,C_2) + w(C),$$
where $t(C_1,C_2)$ is the number of (full) twists of $C_1$ and $C_2$.

We invoke two important techniques of constructing knots to analyze $\pi_1(K_j)$ ($j\geq 2$).  The first is the twisted Whitehead doubling.

\begin{definition}\label{Def: Whitehead doubling}
Let $K_P$ be a knot in $S^3$ and $V_P$ an unknotted solid torus in $S^3$ with $K_P\subset V_P \subset S^3$. Assume $K_P$ is not contained in a 3-ball of $V_P$. Let $K_C \subset S^3$ be another knot and let $V_C$ be a tubular neighborhood of $K_C$ in $S^3$. A homeomorphism $h: V_P \to V_C$ which maps a meridian of $S^3 \backslash \operatorname{Int}V_P$ onto a longitude of $V_C$ and $K_P$ onto a knot $K_W=h(K_P)$. We say $K_C$ is a \emph{companion} of any knot $K_W$ constructed (up to knot type) in this manner, and $K_W$ is called a \emph{satellite} of $K_C$. If $h$ is \emph{faithful}, meaning that $h$ takes the preferred longitude and meridian of $V_P$, respectively, to the preferred longitude and meridian of $V_C$, consider $K_P \subset V_P$ as in Figure \ref{whiteheaddouble_3_1_untwisted},
we say $K_W$ is an \emph{untwisted Whitehead double} of $K_C$. By unclasping and reconnecting $K_W$ shown in Figure \ref{whiteheaddouble_3_1_untwisted}, we obtain a link, which is the disjoint union of two boundary components of a closed ribbon. The "preferred longitude" means that the link number of those two components is zero. If $K_W$ contains no twists, by the White formula, this is equivalent to say that $K_C$ has writhe zero. Note that a Whitehead doubling only consists of three actions: clasp (\raisebox{-.2\height}{\includegraphics[height=3ex,width=3ex]{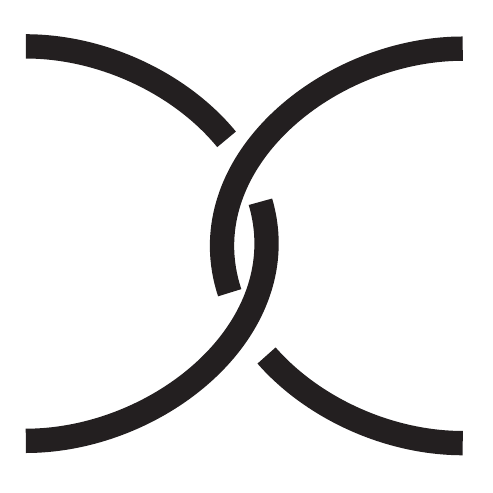}}), twist (\raisebox{-.2\height}{\includegraphics[height=2.8ex,width=10ex]{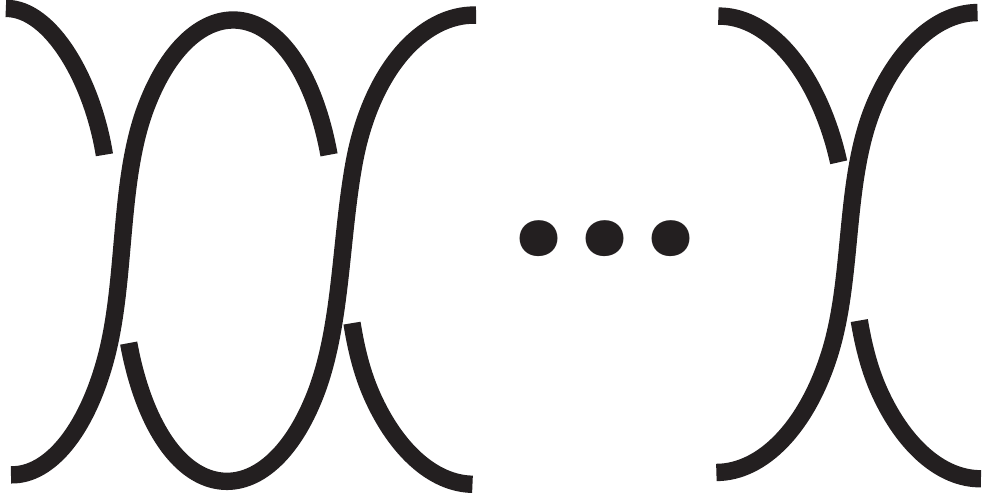}} or \raisebox{-.2\height}{\includegraphics[height=2.8ex,width=10ex]{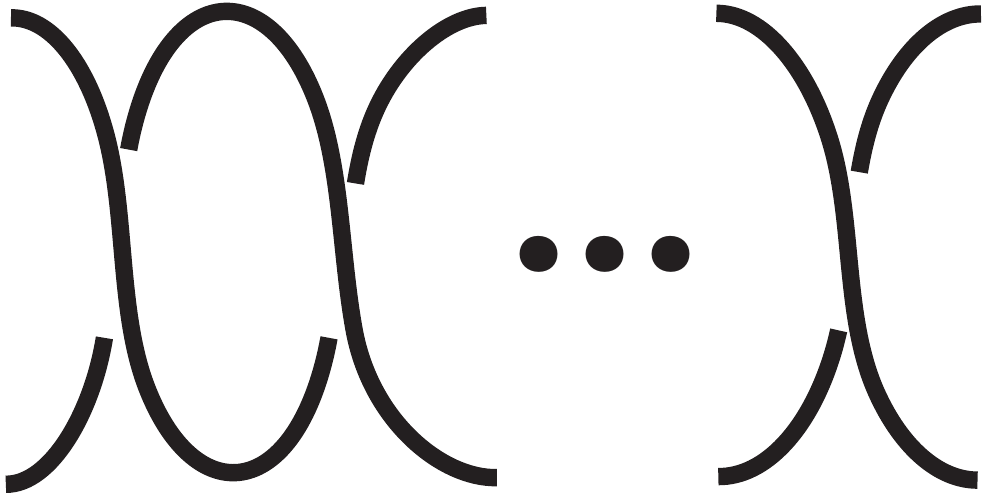}})
and doubling of the crossings (\raisebox{-.2\height}{\includegraphics[height=3.5ex,width=3.5ex]{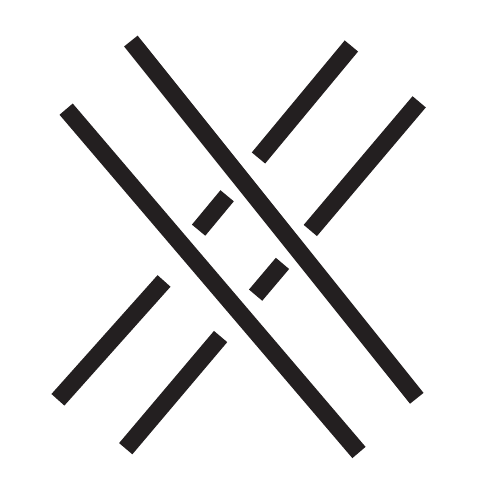}} or \raisebox{-.2\height}{\includegraphics[height=3.5ex,width=3.5ex]{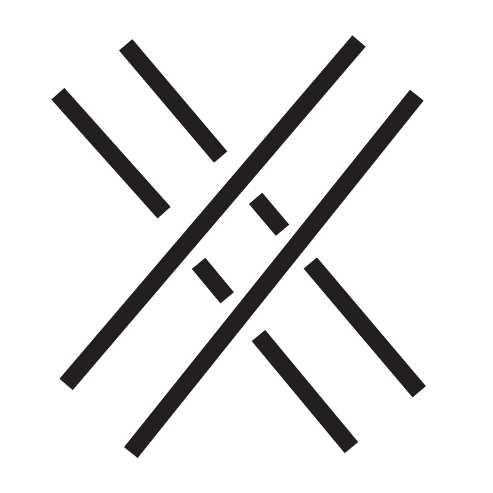}}) of $K_C$. In this situation, the first two do not contribute to the link number. Instead of using the White formula, one can obtain the same conclusion through such an observation. In the notion of Rolfsen \cite[P. 166]{Rol76}, the link number is also known as \emph{the twisting number}. If the linking number is nonzero, then $K_W$ is a \emph{twisted Whitehead double}. For instance, the satellite knot in Figure \ref{whiteheaddouble_3_1} is a 3-twisted Whitehead double of a trefoil knot, where the twisting arises from the writhe of a trefoil knot. In this case, $h$ is not faithful. The pair $(V_P, K_P)$ is called
\emph{pattern} of $K_W$. 
\end{definition}

\begin{figure}[h!]
        \centering
\begin{subfigure}{0.4\textwidth}
       \centering
       \includegraphics[width=0.9\textwidth]{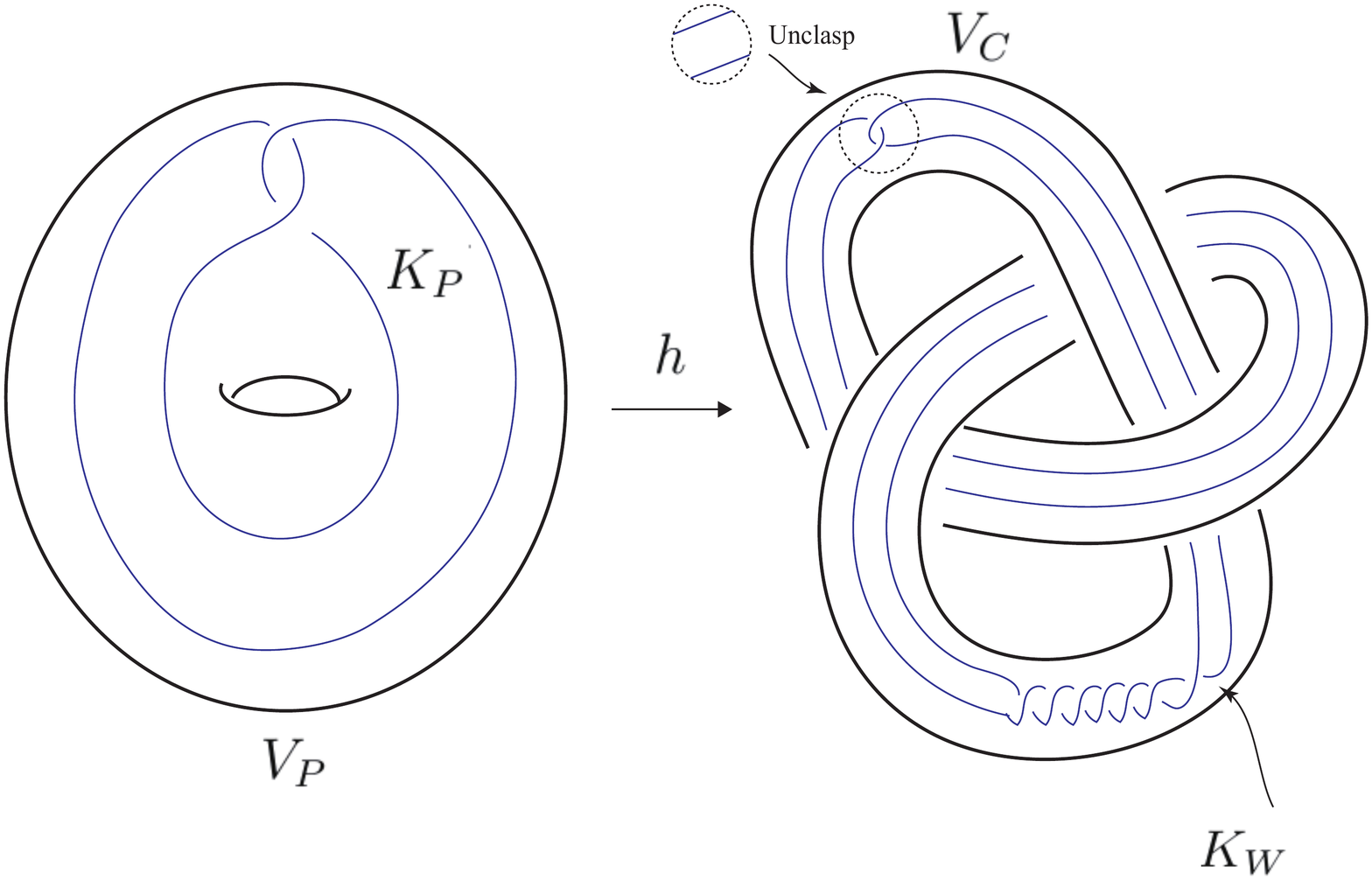}
       \caption{An untwisted Whitehead double of a trefoil knot.}
        \label{whiteheaddouble_3_1_untwisted}
\end{subfigure}\hspace{2em}
\begin{subfigure}{0.4\textwidth}
       \centering
       \includegraphics[width=0.9\textwidth]{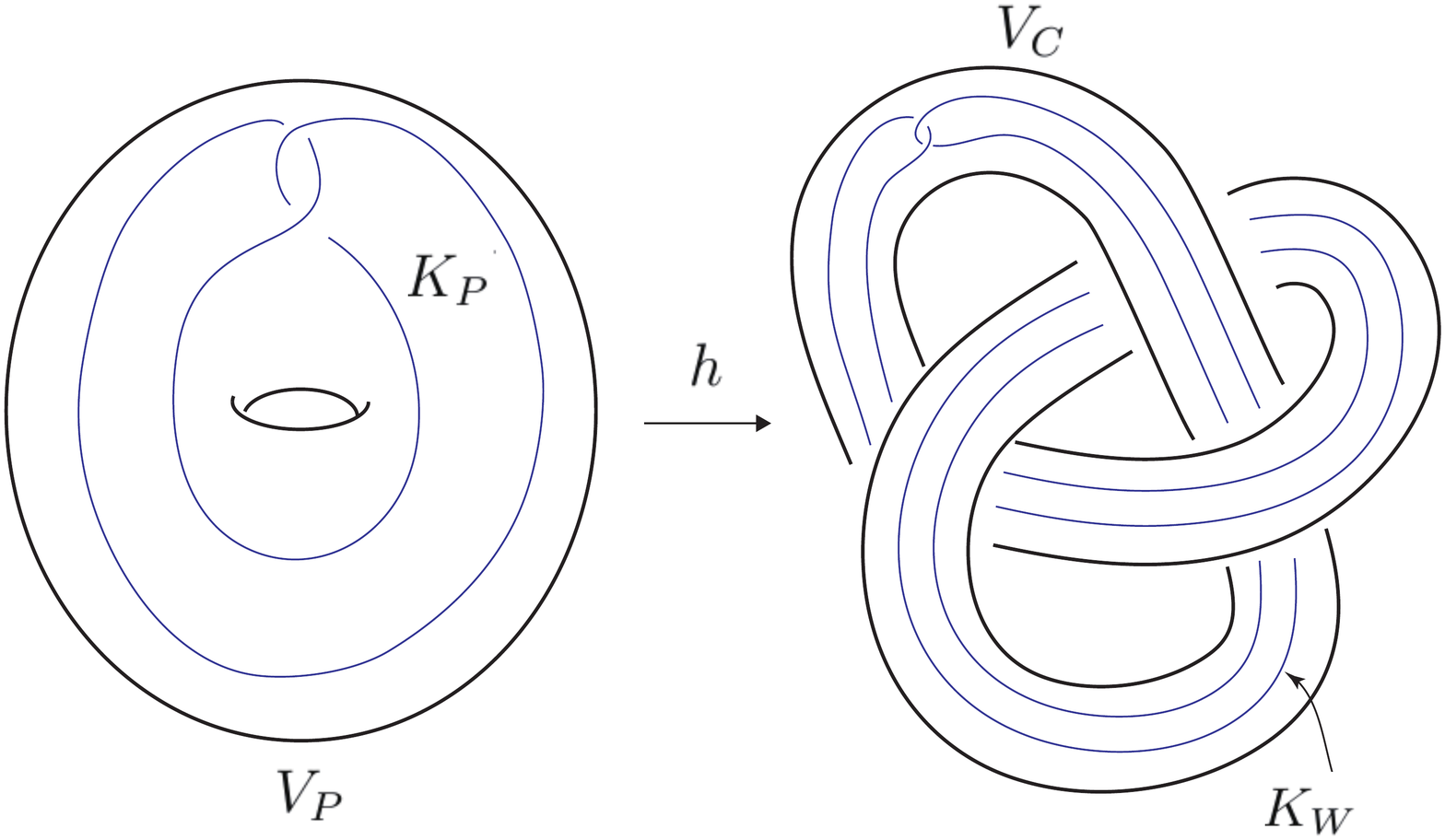}
       \caption{A 3-twisted Whitehead double of a trefoil knot.}
        \label{whiteheaddouble_3_1}
\end{subfigure}
\caption{A knot $K_W$ with a trefoil knot as companion.}
\label{A knot K_W with a trefoil knot as companion}
\end{figure}
The second tool is based on a type of connected sum of a pair of manifolds $(M_{1}^{m},N_{1}^{n})\#(M_{2}^m,N_{2}^{n})$, where $N_{i}^{n}$ is a locally flat submanifold of $M_{i}^{m}$. Treat the above pair as $(S^3, k_1) \# (S^3,k_2)$ where $k_i$ are tame knots. Removing a standard ball pair $(B_i^3,B_i^1)$ from $(S^3,k_i)$ and gluing the resulting pairs by a homeomorphism $h: (\partial B_2^3,\partial B_2^1) \to (\partial B_1^3,\partial B_1^1)$ to form the pair connected sum. For convenience, we use $k_1 \# k_2$ other than pairs of manfolds. See \cite[P. 40]{Rol76} for details. Let $\mathcal{K}_1$ ($=K$) be a knot corresponding to the knot space $K_1$, and denote a twisted Whitehead double of $\mathcal{K}_1$ by $\mathcal{K}_1^{Wh}$. Let $\mathcal{K}_2 = \mathcal{K}_1^{Wh} \# \mathcal{K}_1$. Then an argument similar to \cite[Claim 3]{Gu21} implies that 
$\pi_1(S^3 \backslash \mathcal{K}_2) \cong \pi_1(K_2)$. Likewise, one can further find a knot $\mathcal{K}_3$ on 3rd stage which is a connected sum of a twisted Whitehead double of $\mathcal{K}_2$ and $\mathcal{K}_1$. By iteration, a knot $\mathcal{K}_j$ can be viewed as $\mathcal{K}_{j-1}^{Wh}\# \mathcal{K}_1$ so that $\pi_1(S^3\backslash\mathcal{K}_j)\cong \pi_1(K_j)$. Using Lemma \ref{Lemma:hypoabelian groups}, Sequence (\ref{pi_1 of W^3}) above not being perfectly
semistable is equivalent to $\pi_1(K_j)$ being hypoabelian, thus, our next step is to show that each knot group in Sequence (\ref{knot groups}) is hypoabelian.

\begin{lemma}\label{Lemma: Whitehead link is hypoabelian}
The group of the link complement of the Whitehead link (which is isomorphic to $\pi_1(V_P\backslash K_P)$ as shown in Figure \ref{A knot K_W with a trefoil knot as companion}) is hypoabelian.
\end{lemma}
\begin{proof}
Consider a (unique) epimorphism $\phi$ from $\pi_1(V_P\backslash K_P)$ to the infinite cyclic group $\langle t \rangle$
sending each meridian of $V_P$ to $t$. Since the Whitehead link is fibered \cite[Ex. 5, p. 338]{Rol76}, the kernel $\ker \phi$ of $\phi$ is isomorphic to the fundamental group of the fiber surface, which is a free group of finite rank, thus, hypoabelian. Then
$\pi_1(V_P\backslash K_P)$ is a ension of two hypoabelian groups, hence, hypoabelian.
\end{proof}

\begin{lemma}\label{Lemma: Satellite knot is hypoabelian}
Let $K_C$ be a non-trivial knot, $K_W$ be a satellite knot and $(V_P,K_P)$ be the pattern. Suppose the knot group of $K_C$ and $\pi_1(V_P \backslash K_P)$ are hypoabelian, and that the Alexander polynomial of $K_W$ is nontrivial. Then the knot group of $K_W$ is hypoabelian.
\end{lemma}

\begin{proof}
Denote the various groups $\pi_1(S^3\backslash K_C)$, $\pi_1(S^3 \backslash K_W)$ and $\pi_1(V_P\backslash K_P)$ by $G_C$, $G_W$ and $H$, respectively. By Seifert-van Kampen theorem, $G_W = H \ast_{\pi_1(\partial V_P)} G_C$. The amalgamated free product of $H$ and $G_C$ with subgroups $\pi_1(\partial V_P) \subset H$ and $\pi_1(\partial V_C)\subset G_C$ identified via a homeomorphsim $h$ as described in Definition \ref{Def: Whitehead doubling}.

Regard $H$ as the fundamental group of the complement of a 2-component link $K_P \cup m_P$, where $m_P$ is a meridian of $V_P$. Let $\lambda,t$ be the homotopy class of the longitude $l_P$ and meridian $m_P$ of $V_P$, respectively, in $H$. Using a Wirtinger presentation after replacing all meridional generators of $m_P$ except $t$ by element of the first commutator subgroup $H^{(1)}$ we obtain 
$$H = \langle t, \tilde{u}_i, \lambda | \tilde{R}_{j}(\tilde{u}_{i}^{t^\nu},\lambda) \rangle,$$
where $\tilde{u}_{i}\in H^{(1)}$, $\tilde{u}_{i}^{t^\nu}=  t^\nu \tilde{u}_{i} t^{-\nu}$, $\nu \in \mathbb{Z}$, $\tilde{R}_{j}(\tilde{u}_{i}^{t^\nu},\lambda)$ represents a  relator list of this Wirtinger presentation; $i,j$ are finite indices. The augmentation $\phi: G_W \to \mathbb{Z}$ induces a homomorphism $\phi: H \to \mathbb{Z}$ by sending $t\to 1$ and $\lambda,\tilde{u}_i \to 0$. Let $\mathcal{R}$ be $\ker \phi$. Apply Reidemeister-Schreier method to see that the presentation of the commutator subgroup $G_{W}^{(1)}$ of $G_W$ can be written as a free product with amalgamation:
\begin{equation}
G_{W}^{(1)} = \mathcal{R} \ast_{F} \left(\prod_{\varrho}^{\ast} t^{\varrho} G_{C}^{(1)} t^{-\varrho} \right),
\end{equation}
where $G_{C}^{(1)}$ is the commutator subgroup of $G_C$ and $\varrho$ ranges over $\mathbb{Z}$. When the winding number $n$ of $K_P$ in $V_P$ is zero, $F$ is a free group of infinite rank generated by $\lambda^{t^\varrho} = t^\varrho \lambda t^{-\varrho}$. Otherwise, $F$ is a free group of rank $n$ and $\varrho = 0, 1, \dots, n-1$.
A detailed proof can be found in \cite[Section 4.12]{BZ03}. By presentations (7) and (8) in \cite[p. 63]{BZ03}, we obtain a presentation of $\mathcal{R}$
$$\mathcal{R} = \langle \tilde{u}_{i}^{t^\varrho}, \lambda^{t^\varrho} | \tilde{R}_{j}^{t^\varrho}(\tilde{u}_{i}^{t^\nu},\lambda)\rangle,$$
where $\tilde{R}_{j}^{t^\varrho}(\tilde{u}_{i}^{t^\nu},\lambda)=  t^\varrho \tilde{R}_{j}(\tilde{u}_{i}^{t^\nu},\lambda) t^{-\varrho}$ is a relator. The second group in the amalgamation $G_{W}^{(1)}$ has a presentation (derived from the presentation of $G_C^{(1)}$ \cite[Presentations (6)\&($6\mu$), p. 63]{BZ03})
$$\prod_{\varrho}^{\ast} t^{\varrho} G_{C}^{(1)} t^{-\varrho} = \langle \hat{u}_{k}^{t^\varrho}, \lambda^{t^\varrho} | \hat{R}_{\iota}^{t^\varrho}(\hat{u}_{k}^{t^\varrho}), (\lambda^{t^\varrho})^{-1}\hat{w}^{t^\varrho}(\hat{u}_{k}^{t^\varrho}) \rangle,$$
where $\hat{u}_k, \hat{w}(\hat{u}_{k}) \in G_{C}^{(1)}$,  $\hat{u}_{k}^{t^\varrho} = t^\varrho \hat{u}_{k} t^{-\varrho}$, $\hat{R}_{\iota}^{t^\varrho}(\hat{u}_{k}^{t^\varrho})=  t^\varrho \hat{R}_{\iota}(\hat{u}_{k}^{t^\varrho}) t^{-\varrho}$ represents a relator,
and $\hat{w}^{t^\varrho}(\hat{u}_{k}^{t^\varrho}) = t^\varrho \hat{w}(\hat{u}_{k}^{t^\varrho})t^{-\varrho}$; $k$ is a finite index.

To see that $G_{W}$ is hypoabelian, it suffices to show that the third derived series of $G_{W}$ is hypoabelian. Note that the derived series of $H$ (resp. $G_{C}$) does not stabilize at the third commutator subgroups $H^{(3)}$ (resp. $G_{C}^{(3)}$). In fact, \cite[Prop. 12.5]{Coc04} implies that $\lambda$ vanishes at $G_{W}^{(3)}$. Dropping $\lambda$ and using generators $\tilde{v}_{\kappa} \in H^{(3)}$ and $\hat{v}_{\tau} \in G_{C}^{(3)}$ (while keeping $\lambda$ in $G_{W}^{(2)}$), we obtain 
$$G_{W}^{(3)} =  \langle \tilde{v}_{\kappa}^{\lambda^\varrho}, \hat{v}_{\tau}^{\lambda^\varrho} | \tilde{r}_{\zeta}^{\lambda^\varrho}(\tilde{v}_{\kappa}^{\lambda^\mu}), \hat{r}_{\varsigma}^{\lambda^\varrho}(\hat{v}_{\tau}^{\lambda^\mu})\rangle,$$
where $\mu \in \mathbb{Z}$,  $\tilde{v}_{\kappa}^{\lambda^\varrho} =\lambda^\varrho  \tilde{v}_{\kappa} \lambda^{-\varrho}$,  $\tilde{r}_{\zeta}^{\lambda^\varrho}(\tilde{v}_{\kappa}^{\lambda^\mu}) = \lambda^{\varrho}\tilde{r}_{\zeta}(\tilde{v}_{\kappa}^{\lambda^\mu})\lambda^{-\varrho}$ represents a relator, $\hat{v}_{\tau}^{\lambda^\mu} = \lambda^\varrho  \hat{v}_{\tau} \lambda^{-\varrho}$ and $\hat{r}_{\varsigma}^{\lambda^\varrho}(\hat{v}_{\tau}^{\lambda^\mu}) = \lambda^{\varrho}\hat{r}_{\varsigma}(\hat{v}_{\tau}^{\lambda^\mu})\lambda^{-\varrho}$ represents a relator; $\kappa, \tau, \zeta$ and $\varsigma$ are finite indices. Since $G_{C}^{(3)}$,
$\lambda G_{C}^{(3)} \lambda^{-1}$, $\dots$, have disjoint sets of generators and so do $H^{(3)}$, $\lambda H^{(3)} \lambda^{-1}$, $\dots$, $G_{W}^{(3)}$ can be rewritten as a free product of free products
$$G_{W}^{(3)} = \left(\prod_{\varrho}^{\ast} \lambda^{\varrho} H^{(3)} \lambda^{-\varrho}\right)\ast \left(\prod_{\varrho}^{\ast} \lambda^{\varrho} G_{C}^{(3)} \lambda^{-\varrho}\right).$$
It follows from the hypotheses that $H^{(3)}$ and $G_{C}^{(3)}$ are hypoabelian. The free product is hypoabelian, for the conjugates of $H^{(3)}$ and $G_{C}^{(3)}$ are hypoabelian.
\end{proof}

\begin{remark}
\begin{enumerate}
\item The requirement that the Alexander polynomial of $K_W$ is nontrivial in Lemma \ref{Lemma: Satellite knot is hypoabelian} cannot be omitted. Removing it might cause counterexamples such as the untwisted Whitehead double of any nontrivial knot. It's well-known that the knot group of an untwisted Whitehead double of any nontrivial knot is perfect.

\item \label{Counterexample} That the Alexander polynomial of $K_W$ is nontrivial alone cannot guarantee that the knot group of $K_W$ is hypoabelian. Here is an example suggested by Professor Ian Agol \cite{Ago17}. Consider the $(p,q)$ cable of $K$, denoted $K_{p,q}$, is a statellite knot with pattern the $(p,q)$ torus knot, $T_{p,q}$. More precisely, $K_{p,q}$ is the image of a torus knot living on the boundary of a tubular neighborhood of $K$. Thus, $p$ is the number of times $K_{p,q}$ traverses the longitudinal direction of $K$, and $q$ is the meridional number. By \cite[Prop. 8.23, p. 121]{BZ03}, the Alexander polynomial of a cable knot is determined by
$$\Delta_{K_{p,q}}(t) = \Delta_{T_{p,q}}(t)\cdot \Delta_{K}(t^p).$$
Let $K$ be an untwisted Whitehead double of a nontrivial knot. Since $K$ has trivial Alexander polynomial,  $G^{(1)} = G^{(2)}$. Therefore, when one takes a cable of this knot, its derived series will agree with that of torus knot $T_{p,q}$. Because the knot group of any $T_{p,q}$ is hypoabelian, the transfinite derived series $G_{p,q}^{(i)}$ ($i \geq 0$) of the knot group of $K_{p,q}$ terminates at the commutator subgroup of the pattern knot at $G_{p,q}^{(\omega)}$, where $\omega$ is the first ordinal. 

\item It's impossible to weaken the hypothese that the knot group of $K_C$ and $\pi_1(V_P \backslash K_P)$ are hypoabelian in Lemma \ref{Lemma: Satellite knot is hypoabelian} to the Alexander polynomials of $K_C$ and $K_P$ are nontrivial. One emulates the example in (\ref{Counterexample}) above to produce a cable of a cable of an untwisted Whitehead double of an nontrivial knot such that the patterns are torus knots.
\end{enumerate}
\end{remark}

The following result answers a question posed in \cite{Gu17} and  \cite[Question 4]{Gu20} in affirmative.
\begin{corollary}\label{Corollary: whitehead double is hypoabelian}
Let $K_C$ be a non-trivial knot. If $\pi_1(S^3\backslash K_C)$ is hypoabelian, then the knot group of a twisted Whitehead double $K_W$ of $K_C$ is hypoabelian.
\end{corollary}
\begin{proof}
An easy exercise in \cite[Ex. 7, p. 166]{Rol76} shows that the Alexander polynomial of $K_W$ is nontrivial. Apply Lemma \ref{Lemma: Whitehead link is hypoabelian} and Lemma \ref{Lemma: Satellite knot is hypoabelian}.
\end{proof}

\begin{lemma}\label{Lemma: Connected sum of knots is hypoabelian}
Let $K_1$ and $K_2$ be knots and $G_1$ and $G_2$ be the corresponding knot groups. If $G_i$ is hypoabelian, then the 
knot group of $K_1 \# K_2$ is hypoabelian.
\end{lemma}
\begin{proof}
Note that $\pi_1(S^3 \backslash K_1 \# K_2)$ can be written as a free product of $G_1$ and $G_2$ amalgmating an infinite cyclic group generated by 
the homotopy class of meridian of $K_1$ and $K_2$. The result follows readily from Proposition \ref{Prop: splitting is hypoabelian}. 
\end{proof}

\begin{proposition}\label{Prop: W3 x [0,1) is not pseudo-collarable}
$W^3 \times [0,1)$ is not peripherally perfectly $\pi_1$-semistable at infinity, hence, not pseudo-collarable.
\end{proposition}
\begin{proof}
Iterating Lemmas \ref{Lemma: Whitehead link is hypoabelian}, \ref{Lemma: Connected sum of knots is hypoabelian}, Corollary \ref{Corollary: whitehead double is hypoabelian} and the hereditary of hypoabelianism under group extension for $\pi_1(K_i)$'s in Sequence (\ref{knot groups}) shows that Sequences (\ref{pi_1 of quotient space of W^3}) and 
(\ref{pi_1 of W^3}) are pro-isomorphic to a sequence of hypoabelian groups. Furthermore, Lemma \ref{Lemma: inward tameness implies semistable} or a direct observation ensures that Sequence (\ref{pi_1 of W^3}) is semistable.  However, Lemma \ref{Lemma:hypoabelian groups} shows that the perfectness fails. Applying the necessity Theorem \ref{Th: Characterization Theorem} completes the proof.
\end{proof}

Let $D$ be a link diagram, and $x$ a crossing of $D$. Associated to $D$ and $x$ are two
link diagrams, each with one fewer crossing than $D$, called the $A$-\emph{resolution} and
$B$-\emph{resolution} of the crossing. See Figure \ref{crossing resolution}. A knot or link diagram is \emph{A-adequate} if, when all the crossings are resolved using $A$-resolution, each resulting simple closed curve runs through the remnants of each crossing at most once. $B$\emph{-adequacy} is
defined similarly using the $B$-resolution at each crossing. A diagram is \emph{adequate} if it is both
$A$- and $B$-adequate. It follows from a direct observation or \cite[Lemma 4]{Pas17} that adequacy is preserved under the "taking parallels" action such as the doubling action in a Whitehead double. An adequate knot is a generalization of an alternating knot. Many nice topological and geometric properties of adequate knots can be found in \cite{Abe09, FKP11, FKP13,RT88}.
\begin{figure}[h!]
        \centering
        \vspace{-3em}
       \includegraphics[width=9cm, height=7cm]{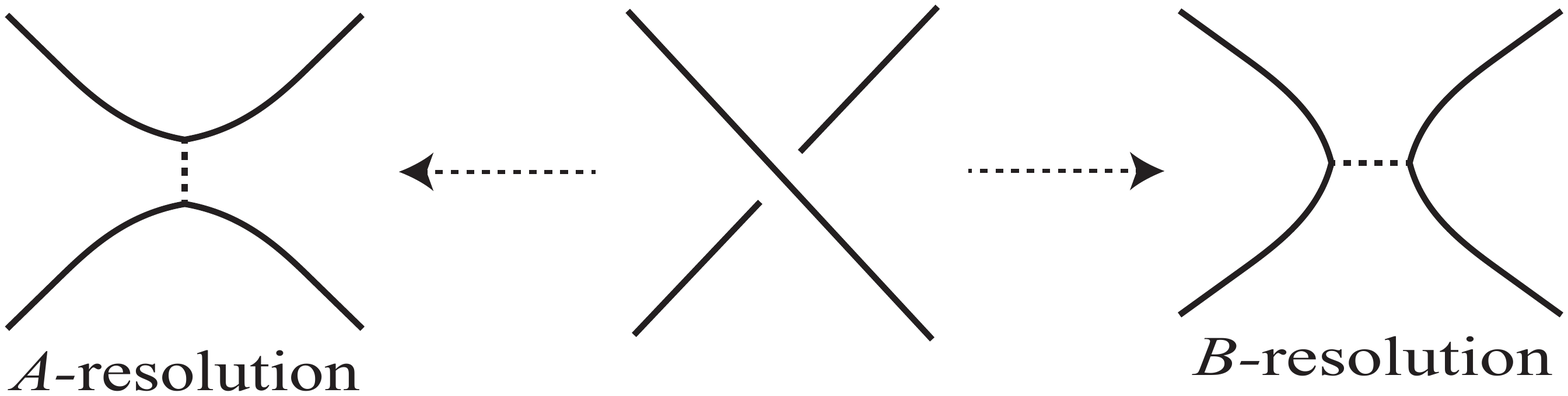}
       \vspace{-3em}
       \caption{$A$- and $B$-resolutions of a crossing.}
        \label{crossing resolution}
\end{figure}

\begin{corollary}\label{Cor: Infinitely many 3D examples}
There are infinitely many $W^3 \times [0,1)$ (up to homeomorphism) which are not pseudo-collarable.
\end{corollary}
\begin{proof}
It suffices to show that there are infinitely many distinct choices (up to ambient isotopy in $S^3$) for the knot $K$ in Figure \ref{3_1knot}. Take a fibered adequate knot $K$ with writhe $\geq 3$, e.g., a trefoil knot.  Note that the connected sum of fibered knots is fibered, thus, the corresponding knot group is hypoabelian. By the definition of adequate knot, the connected sum of fibered adequate knots is fibered adequate, one can produce an infinite sequence $\{\mathcal{K}_i\}$ of distinct fibered adequate knots by taking the $i$-fold connected sum of copies of $K$. We further employ the Tait conjectures \cite{Kau90, Mur87a, Mur87b, Thi87} to assure that the writhe of each $\mathcal{K}_i$ is greater than or equal to $3i$. Next, we shall show that $W^3 \times [0,1)$ constructed based on distinct pair $(\mathcal{K}_i,\mathcal{K}_j)$ $(i \neq j)$ are not homeomorphic. Note that each compact subset of $W^3$ embeds in $S^3$. For a cube with handles $R_l \subset W^3$ at stage $l$ in the construction, we take a homeomorphism of $S^3$ onto itself taking $R_l$ onto a canonical position (e.g., see Figure \ref{fig reembedding}). Since the Tait conjectures ensure the additivity of crossing numbers when taking connected sum, and the crossing number is a topological invariant, there is no ambient isotopy which moves $\mathcal{K}_i$ to $\mathcal{K}_j$ in $S^3$. Hence, the $W^3 \times [0,1)$ constructed based on $\mathcal{K}_i$ is not homeomorphic to the one constructed using $\mathcal{K}_j$.
\end{proof}

\begin{figure}[h!]
        \centering
       \includegraphics[width=10cm, height=8cm]{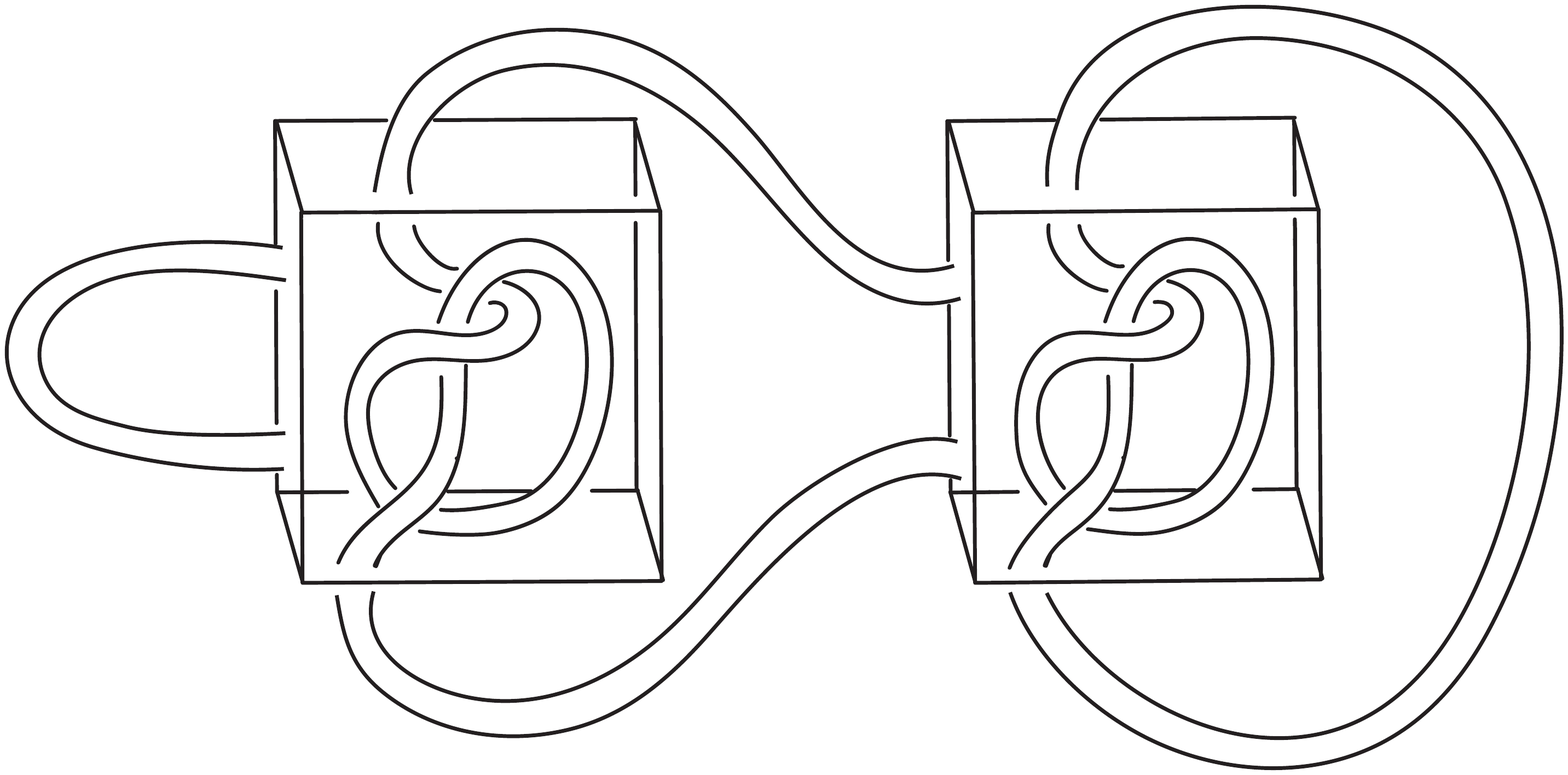}
       \caption{Reembedding of $R_l$ in $S^3$ based on a trefoil knot.}
        \label{fig reembedding}
\end{figure}

\subsection{High-dimensional examples}\label{subsection: High-dimensional examples}

%
%
%
%
%
%
In this section, we shall produce examples of higher dimensions modeled on $W^3 \times [0,1)$ constructed in \S \ref{subsection: pi_1 at infinity of W}.
\begin{proof}[Proof of Theorem \ref{Th: Z-compact not pseudo}]
The case $m = 4$ is just Propositions \ref{Prop: contractible is Z-compact} and \ref{Prop: W3 x [0,1) is not pseudo-collarable}. 

For cases $m >4$, let's let $W^3$ be as in \S \ref{subsection: pi_1 at infinity of W}. For convenience,
we denote $W^3 \times [0,1)$ by $X$. We shall show that $X \times S^{(m-4)}$ is $\mathcal{Z}$-compactifiable but not pseudo-collarable.
 Let $\hat{X}= X \sqcup \{\infty\}$ be the one-point compactification of $X$. By Proposition \ref{Prop: contractible is Z-compact}, $\hat{X}$ is a $\mathcal{Z}$-compactification of $X$. We invoke an alternative definition of $\mathcal{Z}$-compactification that there exists a homotopy $H: [0,1] \times \hat{X} \to \hat{X}$ such that $H_{0} = \operatorname{Id}_{\hat{X}}$ 
and $H_{t}(\hat{X})\subseteq \hat{X} - \{\infty\}$ for all $t > 0$. Subsequently, we define a homotopy $H' = H \times \operatorname{Id}:   [0,1] \times \hat{X}\times S^{(m-4)} \to \hat{X} \times S^{(m-4)}$
such that $H'_{0} = \operatorname{Id}_{\hat{X}\times S^{(m-4)}}$ and $H'_{t}(\hat{X} \times S^{(m-4)})\subseteq (\hat{X} - \{\infty\})\times S^{(m-4)}$ for all $t > 0$.
Thus, $\hat{X} \times S^{(m-4)}$ is a $\mathcal{Z}$-compactification of $X \times S^{(m-4)}$.

To see the breakdown of pseudo-collarability, we will consider the fundamental group at infinity of $X \times S^{(m-4)}$. More specifically, it suffices to show that the peripherally perfect $\pi_1$-semistability fails. One just need to consider the same cofinal sequence (\ref{pi_1 of W^3}) of neighborhoods of infinity as depicted in Section \ref{subsection: pi_1 at infinity of W}, and take the product of Sequence (\ref{pi_1 of W^3})  with $S^{(m-4)}$. That gives the sequence (with baserays surpressed) as follows.

\begin{equation}
\pi_{1}\left(  (N_0 \cup V_{1})\times S^{(m-4)}   \right)  {\longleftarrow}\pi_{1}\left(  (N_0 \cup V_{2}) \times S^{(m-4)}  \right)
{\longleftarrow}\pi_{1}\left(  (N_0\cup V_{3}) \times S^{(m-4)}\right) {\longleftarrow}\cdots
\label{pi_1 of W^3 x S^n}
\end{equation}

For $m \geq 6$, Sequence (\ref{pi_1 of W^3 x S^n}) is pro-isomorphic to Sequence (\ref{pi_1 of W^3}), which is an inverse sequence of hypoabelian groups. Hence, $X \times S^{(m-4)}$ is not peripherally perfectly $\pi_1$-semistable,
using Lemma \ref{Lemma:hypoabelian groups}.

For $m = 5$, Sequence (\ref{pi_1 of W^3 x S^n}) is pro-isomorphic to

\begin{equation}\label{hypoabelian sequence2}
\pi_{1}(N_0 \cup V_{1}) \times \mathbb{Z}  {\longleftarrow}\pi_{1}(N_0 \cup V_{2}) \times \mathbb{Z}
{\longleftarrow}\pi_{1}(N_0\cup V_{3})\times \mathbb{Z} {\longleftarrow}\cdots
\end{equation}
By the work in \S \ref{subsection: pi_1 at infinity of W}, $\pi_{1}(N_0 \cup V_{j})$ and $\mathbb{Z}$ are hypoabelian, therefore, Sequence (\ref{hypoabelian sequence2}) is an inverse sequence of hypoabelian groups. Apply Lemma \ref{Lemma:hypoabelian groups} and the necessity of Theorem \ref{Th: Characterization Theorem}.
\end{proof}


\section*{Acknowledgements}
I would like to thank Professor Craig Guilbault for introducing hypoabelian groups to me, and his marvelous instinct that $W^3 \times [0,1)$ is a counterexample. I also thank the anonymous referee for very helpful and constructive comments, especially, for pointing out the error in the proof of Proposition \ref{Prop: free product of hypoabelian} of the initial version of the paper.


\begin{thebibliography}{9}


\bibitem[Abe09]{Abe09}Tetsuya Abe, \emph{The Turaev genus of an adequate knot}, Topology Appl., \textbf{156} (2009), no. 17, 2704--2712.

\bibitem[Ago17]{Ago17} Ian Agol (https://mathoverflow.net/users/1345/ian-agol), Does there exist a limit of (infinite) iterated universal abelian covers of a knot group,  https://mathoverflow.net/q/281109 (version: 2017-09-14).

\bibitem[AS85]{AS85} F.D. Ancel, L.C. Siebenmann, \emph{The construction of homogeneous homology manifolds}, Abstracts
Amer. Math. Soc. \textbf{6} (1985), abstract 816-57-72.

\bibitem[BM91]{BM91} Mladen Bestvina and Geoffrey Mess, \emph{The boundary of negatively curved groups}, J. Amer.
Math. Soc. \textbf{4} (1991), no. 3, 469--481. MR 1096169

\bibitem[BZ03]{BZ03} G. Burde and H. Zieschang, \emph{Knots}, de Gruyter Studies in Mathematics, Number 5,
Walter de Gruyter and Co., Berlin-New York, 2003.

\bibitem[Cal61]{Cal61} G. C\v{a}lug\v{a}reanu, \emph{Sur les classes d'isotopie des noeuds
tridimensionnels et leurs invariants}, Czech. Math. J. \textbf{11} (1961),
588--625.



\bibitem[Coc04]{Coc04} Tim D. Cochran, \emph{Noncommutative knot theory}, Algebraic Geom. Topol. \textbf{4} (2004), 347--398.

\bibitem[CS76]{CS76}
T. A. Chapman and L. C. Siebenmann, \emph{Finding a boundary for a Hilbert cube manifold},
Acta Math. \textbf{137} (1976), no. 3--4, 171--208. MR 0425973



\bibitem[Dav83]{Dav83} M. W. Davis, \emph{Groups generated by reflections and aspherical manifolds not covered
by Euclidean space}, Ann. Math. \textbf{117} (1983) 293--325.



\bibitem[FKP11]{FKP11} David Futer, Efstratia Kalfagianni, and Jessica S. Purcell, \emph{Slopes and colored Jones polynomials of adequate knots}, Proc. Amer. Math. Soc., \textbf{139}(2011), 1889--1896.

\bibitem[FKP13]{FKP13} \bysame, Guts of surfaces and the colored
Jones polynomial, vol. 2069 of Lecture Notes in Mathematics, Springer, Heidelberg, 2013.

\bibitem[Fis03]{Fis03} Hanspeter Fischer, \emph{Boundaries of right-angled Coxeter groups with manifold nerves}, Topology, \textbf{42}(2003), no. 2, 423--446.

\bibitem[FQ90]{FQ90} Michael H. Freedman and Frank Quinn, Topology of 4-manifolds, Princeton Mathematical
Series, vol. 39, Princeton University Press, Princeton, NJ, 1990. MR 1201584



\bibitem[Ful71]{Ful71} F. B. Fuller, \emph{The writhing number of a space curve}, Proc.
Natl. Acad. Sci. USA \textbf{68} (1971), 815--819.




\bibitem[Geo08]{Geo08} R. Geoghegan, \emph{Topological methods in group theory}, Graduate Texts in Mathematics, vol.
243, Springer, New York, 2008. MR 2365352


\bibitem[GG20]{GG20} Shijie Gu and Craig R. Guilbault, \emph{Compactifications of manifolds with boundary}, J. Topol. Anal. \textbf{12} (2020), no. 04, 1073--1101.

\bibitem[GT03]{GT03} C. R. Guilbault and F. C. Tinsley, \emph{Manifolds with non-stable fundamental groups at in-
finity. II}, Geom. Topol. \textbf{7} (2003), 255--286. MR 1988286


\bibitem[GT06]{GT06} \bysame, \emph{Manifolds with non-stable fundamental groups at infinity. III}, Geom. Topol. \textbf{10}
(2006), 541--556. MR 2224464

\bibitem[GS99]{GS99} N. Gupta and S. Sidki, \emph{On torsion-free metabelian groups with commutator quotients of prime exponent}, Internat. J. Algebra Comput. \textbf{9} (1999), no. 5, 493--520.

\bibitem[Gu17]{Gu17} Shijie Gu (https://mathoverflow.net/users/114032/shijie-gu), Twisted Whitehead double of trefoil knot. MathOverflow. https://mathoverflow.net/q/280073 (version: 2017-09-13).

\bibitem[Gu20]{Gu20} \bysame, \textit{Characterization of pseudo-collarable manifolds with boundary}, Michigan Math. J. \textbf{69} (2020), no. 4, 733--750.

\bibitem[Gu21]{Gu21} \bysame, \emph{Contractible open manifolds which embed in no compact, locally connected and locally $1$-connected metric space}, Algebra. Geom. Topol.  \textbf{21} (2021), no. 3, 1327--1350.

\bibitem[Gui00]{Gui00} Craig R. Guilbault, \emph{Manifolds with non-stable fundamental groups at infinity}, Geom. Topol. \textbf{4} (2000), 537--579. MR 1800296

\bibitem[Gui01]{Gui01} \bysame, \emph{A non-$\mathcal{Z}$-compactifiable polyhedron whose product with the Hilbert cube
is Z-compactifiable}, Fund. Math. \textbf{168} (2001), no. 2, 165--197. MR 1852740

\bibitem[Gui16]{Gui16} \bysame, \emph{Ends, shapes, and boundaries in manifold topology and geometric group theory},
Topology and geometric group theory, Springer Proc. Math. Stat., vol. 184, Springer,
[Cham], 2016, pp. 45--125. MR 3598162

\bibitem[Hak68]{Hak68} W. Haken, \emph{Some results on surfaces in $3$-manifolds}, Studies in modern topology (M.A.A., Prentice-Hall, 1968), 39--98.




\bibitem[How79]{How79} J. Howie, \emph{Aspherical and acyclic $2$-complexes}, J. London Math. Soc. \textbf{20} (1979), 549--558.

\bibitem[Hu65]{Hu65} Sze-tsen Hu, \emph{Theory of retracts}, Wayne State University Press, Detroit, 1965.


\bibitem[JR16]{JR16} Eric Jespers, \'Angel del R\'io, \emph{Group ring groups}, De Gruyter Textbook, vol. 2, Walter de Gruyter GmbH, Berlin/Boston, 2016.

\bibitem[Kah11]{Kah11} D. Kahrobaei, \emph{Residual Solvability of generalized free products of finitely generated nilpotent
groups}, Comm. Algebra, \textbf{Vol. 39} Issue 2, (2011), 649--658.


\bibitem[Kan79]{Kan79} T. Kanenobu, \emph{The augmentation subgroup of a pretzel link}, Math. Sem. Notes Kobe Univ. \textbf{7}
(1979), 363--384.

\bibitem[Kau90]{Kau90}L. H. Kauffman, \emph{An invariant of regular isotopy}, Trans. Amer. Math. Soc. \textbf{318} (1990),
417–471. MR958895 (90g:57007)

\bibitem[KM62]{KM62} J. M. Kister and D. R. McMillan, Jr., \emph{Locally Euclidean factors of $\mathbb{E}^4$ which cannot be embedded in $\mathbb{E}^3$}, Ann. of Math. \textbf{76} (1962), 541--546.

\bibitem[KM11]{KM11}
Delaram Kahrobaei, Stephen Majewicz, \emph{On the residual solvability of generalized free products of
solvable groups}, DMTCS, 2011, Vol. 13 no. 4 (4), pp. 45--50.

\bibitem[KS88]{KS88} Slawomir Kwasik and Reinhard Schultz, \emph{Desuspension of group actions and the ribbon
theorem}, Topology \textbf{27} (1988), no. 4, 443--457. MR 976586











\bibitem[Mur63]{Mur63} K. Murasugi, \emph{On a certain subgroup of the group of an alternating link}, Amer. J. Math. \textbf{85}
(1963), 544--550.

\bibitem[Mur87a]{Mur87a}K. Murasugi, \emph{Jones polynomial and classical conjectures in knot theory}, Topology \textbf{26}
(1987), 187--194. 

\bibitem[Mur87b]{Mur87b}K. Murasugi, \emph{Jones polynomials and classical conjectures in knot theory II}, Math. Proc. Cambridge Philos. Soc. \textbf{102}(2) (1987), 317--318. MR898151

\bibitem[Neu49]{Neu49} H. Neumann, \emph{Generalized free producs with amalgamated subgroups}, Amer. J. Math. \textbf{71}
(1949), 491--540.

\bibitem[Neu65]{Neu65} L.P. Neuwirth, \emph{Knot groups}, Annals of Mathematics Studies, No. 56 Princeton University Press, Princeton, N.J. 1965 vi+113 pp. MR 0176462


\bibitem[O'B83]{O'B83} G. O'Brien, \emph{The missing boundary problem for smooth manifolds of dimension greater
than or equal to six}, Topology Appl. \textbf{16} (1983), no. 3, 303--324. MR 722123




\bibitem[Pas17]{Pas17} A. J. Pascual, \emph{On wrapping number, adequacy and the crossing number of satellite knots}, 	arXiv:1712.05635.

\bibitem[Poh68]{Poh68}W. F. Pohl, \emph{The self-linking number of a closed space curve},
J. Math. Mech. \textbf{17}(1968), 975--985.



\bibitem[Rap60]{Rap60} E. S. Rapaport, \emph{On the commutator subgroup of a knot group},
Ann. Math. \textbf{71} (1960), 157--162.


\bibitem[Rob95]{Rob95} D. J. S. Robinson, \emph{A course in the theory of groups}, Graduate Texts in Mathematics, vol. 80, Springer-Verlag, New York, 1996.

\bibitem[Rol76]{Rol76} D. Rolfsen, \emph{Knots and Links}, Publish or Perish Press, Berkeley, CA, 1976.

\bibitem[RT88]{RT88} W. B. Raymond Lickorish and Morwen B. Thistlethwaite,  \emph{Some links with nontrivial polynomials and their crossing-numbers}, Comment. Math. Helv., \textbf{63}(1988), no. 4, 527--539.



\bibitem[Sie65]{Sie65} Laurence Carl Siebenmann, \emph{The obstructuon to finding a boundary for an open manifold
of dimension greater than five}, ProQuest LLC, Ann Arbor, MI, 1965, Thesis (Ph.D.)-Princeton University. MR 2615648

\bibitem[Sta61]{Sta61} J. Stallings, \emph{On fibering certain 3-manifolds}, 1962 Topology of $3$-manifolds and related topics
(Proc. The Univ. of Georgia Institute, 1961), 95--100. 



\bibitem[Thi87]{Thi87}  M. B. Thistlethwaite, \emph{A spanning tree expansion of the Jones polynomial}, Topology \textbf{26}(3) (1987),
297--309. 

\bibitem[Tuc74]{Tuc74} Thomas W. Tucker, \emph{Non-compact 3-manifolds and the missing-boundary problem}, Topology
\textbf{13} (1974), 267--273. MR 0353317

\bibitem[Wal65]{Wal65} C. T. C. Wall, \emph{Finiteness conditions for CW-complexes}, Ann. of Math. (2) \textbf{81} (1965),
56--69. MR 0171284

\bibitem[Wei87]{Wei87} S. Weinberger, \emph{On fibering four- and five-manifolds}, Israel J. Math. \textbf{59} (1987), no. 1,
1--7. MR 923658


\bibitem[Whi69]{Whi69}J. White, \emph{Self-linking and the Gauss integral in higher dimensions},
Amer. J. Math. \textbf{XCI} (1969), 693--728.


%
%
%


\end{thebibliography}
\end{document}